 \documentclass[final]{siamltex}
\usepackage{amsfonts}
\usepackage{amsmath}
\usepackage{amssymb}
\usepackage{array}
\usepackage{color}
\usepackage{eucal}
\usepackage{framed}
\usepackage{graphicx}
\usepackage{mathrsfs}
\usepackage{relsize}
\usepackage{tikz}
\usepackage{multirow}
\usepackage{url}

\usepackage[notref,notcite]{showkeys}

\numberwithin{equation}{section}
\numberwithin{table}{section}
\numberwithin{figure}{section}
\newcommand{\Oh}{{\mathcal{T}_h}}
\newcommand{\Eh}{\mathcal{E}_h}

\newcommand{\dK}{{\partial K}}

\newcommand{\Vh}{\boldsymbol{V}_h}
\newcommand{\Qh}{\mathring{{Q}_h}}

\newcommand{\Mh}{M_h}
\newcommand{\MMh}{\bld M_h}
\newcommand{\VV}{{\boldsymbol V}}
\newcommand{\VVdiv}{{\boldsymbol V}_h^\mathrm{div}}
\newcommand{\MM}{{\bld{M}}}

\newcommand{\Mth}{{\bld{M}}^t_h}
\newcommand{\Mnh}{{\bld{M}}^n_h}

\newcommand{\GG}{{\mathcal{G}}}

\newcommand{\vv}{\bld{v}}

\newcommand{\wwhat}{\widehat{w}}

\newcommand{\uthat}{\widehat{\bld u}_t}
\newcommand{\unhat}{\widehat{\bld u}_{n}}
\newcommand{\vthat}{\widehat{\bld v}_t}
\newcommand{\zthat}{\widehat{\bld z}_t}
\newcommand{\zzhat}{\widehat{\bld z}}

\newcommand{\Pim}{P_M}

\newcommand{\egg}{\mathrm{e}_{L}}
\newcommand{\euu}{\bld{e}_{u}}
\newcommand{\euuhat}{\bld{e}_{\widehat u_t}}
\newcommand{\epp}{{e}_{p}}
\newcommand{\dgg}{\mathrm{\delta}_{L}}
\newcommand{\duu}{\bld{\delta}_{u}}
\newcommand{\duuhat}{\bld{\delta}_{\widehat u_t}}
\newcommand{\dpp}{{\delta}_{p}}

\newcommand{\delPhi}{\delta_{\Phi}}
\newcommand{\delphi}{\bld \delta_{\phi}}
\newcommand{\delvphi}{\delta_{\varphi}}

\newcommand{\mg}{\mathrm{g}}
\newcommand{\mr}{\mathrm{r}}
\newcommand{\QQ}{{Q}}
\newcommand{\qq}{q}
\newcommand{\BDM}{\bld{\mathrm{BDM}}}
\newcommand{\RT}{\bld{\mathrm{RT}}}
\newcommand{\BDFM}{\bld{\mathrm{BDFM}}}

\newcommand{\vvhat}{\widehat{\bld{v}}}

\newcommand{\Pigg}{P_\GG}

\newcommand{\Pimm}{P_{\bld{M}^t}}
\newcommand{\Piww}{\Pi_{\VV}}
\newcommand{\Piq}{P_{Q}}

\newcommand{\bint}[2]{( #1\,,\,#2 )_{\Oh}}
\newcommand{\bintKK}[2]{( #1\,,\,#2 )_{K}}
\newcommand{\bintEh}[2]{\langle #1\,,\,#2 \rangle_{\partial{\Oh}}}
\newcommand{\bintK}[2]{\langle #1\,,\,#2 \rangle_{\partial{K}}}
\newcommand{\bintF}[2]{\langle #1\,,\,#2 \rangle_{F}}

\newcommand{\inp}[2]{(#1\,, \, #2)_{\Oh}}
\newcommand{\norm}[1]{\|#1\|_{\Oh}}
\newcommand{\normEh}[1]{\|#1\|_{\partial\Oh}}
\newcommand{\normh}[1]{\|#1\|_{1,\Omega}}

\newcommand{\tr}{\mathrm{tr}}

\newcommand{\n}{\boldsymbol{n}}
\newcommand{\bld}[1]{\boldsymbol{#1}}
\newcommand{\bR}{\mathbb R}
\newcommand{\jmp}[1]{\,[\![#1]\!]}
\newcommand{\pol}{\EuScript{P}}
\newcommand{\bpol}{\boldsymbol{\pol}}

\newcommand{\curls}{{{\nabla\times}}}
\newcommand{\divs}{{\nabla\cdot}}
\newcommand{\grads}{{\nabla}}
\newcommand{\divv}{{\bld{\nabla\cdot}}}
\newcommand{\gradv}{{\bld{\nabla}}}

\newcommand{\ml}{\mathrm{L}}

\newcommand{\vertiii}[1]{{\left\vert\kern-0.25ex\left\vert\kern-0.25ex\left\vert #1
    \right\vert\kern-0.25ex\right\vert\kern-0.25ex\right\vert}}

\definecolor{red}{rgb}{1,0,0}

\title{Parameter-free superconvergent $H(\mathrm{div})$-conforming HDG methods
for the Brinkman equations.}

\author{
Guosheng Fu
\thanks{Division of Applied Mathematics, Brown University, Providence, %
  RI 02912, USA,  email: {\tt guosheng\_fu@brown.edu}.}
\and
Yanyi Jin
        \thanks{Department of Mathematics, City University of Hong Kong,
              83 Tat Chee Avenue, Kowloon, Hong Kong, China, email: {\tt yyjin2-c@my.cityu.edu.hk}.}
\and
Weifeng Qiu
\thanks{Corresponding author. Department of Mathematics, City University of Hong Kong,
              83 Tat Chee Avenue, Kowloon, Hong Kong, China, email: {\tt weifeqiu@cityu.edu.hk}.
              The work of Weifeng Qiu was partially supported by a grant from the Research Grants
              Council of the Hong Kong Special Administrative Region, China (Project No. CityU 11302014).}
}

\begin{document}

\maketitle

\begin{abstract}
In this paper, we present new parameter-free {superconvergent} $H(\mathrm{div})$-conforming HDG methods for 
the Brinkman equations on both simplicial and rectangular meshes. 
{The methods are based on a velocity gradient-velocity-pressure formulation, which can be considered as a natural 
extension of the $H(\mathrm{div})$-conforming HDG method (defined on simplicial meshes) for the Stokes flow 
[Math. Comp. 83(2014), pp. 1571-1598].}

We obtain optimal  $L^2$-error estimate for the velocity in both the Stokes-dominated regime (high viscosity/permeability ratio) and Darcy-dominated regime (low viscosity/permeability ratio). 
We also obtain superconvergent $L^2$-estimate of one order higher for
a suitable projection of the velocity error in the Stokes-dominated regime.
Moreover, thanks to $H(\mathrm{div})$-conformity of the velocity, our velocity error estimates are independent of the pressure regularity. {
Furthermore, we  provide a discrete $H^{1}$-stability result of the velocity field, which is essential
in the error analysis of the natural generalization of these new HDG methods to the incompressible Navier-Stokes equations.}

Preliminary numerical results on both triangular and rectangular meshes in two dimensions confirm our theoretical predictions.
\end{abstract}



\thispagestyle{plain} \markboth{G. Fu, et.~al.}{
$H(\mathrm{div})$-conforming HDG for the Brinkman equations}


\section{Introduction}
\label{sec:intro}
In this paper, we devise superconvergent
$H(\mathrm{div})$-conforming hybridizable discontinuous Galerkin (HDG) method for the following Brinkman equations in velocity gradient-velocity-pressure formulation:
\begin{subequations}
\label{nb}
\begin{align}
\label{nba}
\ml=\nabla\boldsymbol{u}&\;\qquad \text{in}\quad\Omega,\\
\label{nbb}
-\nu\nabla\cdot \ml+\gamma\boldsymbol{u}+\nabla p=\boldsymbol{f}&\;\qquad\text{in}\quad\Omega,\\
\label{nbc}
\nabla\cdot\boldsymbol{u}=g&\;\qquad \text{in}\quad \Omega,\\
\label{nbd}
\boldsymbol{u}\cdot\boldsymbol{n}=0&\;\qquad \text{on}\quad \partial\Omega,\\
\label{nbe}
\nu(I_{d}-\boldsymbol{n}\otimes\boldsymbol{n})\boldsymbol{u}=0&\;\qquad \text{on}\quad \partial\Omega,\\
\label{nbf}
\int_\Omega p=0&\;,
\end{align}
\end{subequations}
where $\ml$ is the velocity gradient,
$\boldsymbol{u}$ is the velocity, $p$ is the pressure, $\nu$ is the effective viscosity constant,
$\gamma\in {L}^{\infty}(\Omega)^{d\times d}$ is inverse of the permeability tensor,
and $\boldsymbol{f}\in{L}^2(\Omega)^d$
is the external body force. The domain $\Omega\subset\mathbb{R}^d$ is a polygon $(d=2)$ or polyhedron $(d=3)$.


One challenging aspect of numerical discretization of the Brinkman equations is the construction of stable finite element methods
in both Stokes-dominated and Darcy-dominated regimes. We refer to such methods as uniformly stable methods.
Uniformly stable methods for the Brinkman equations have been extensively studied for the classical velocity-pressure formulation, including the nonconforming methods with an $H(\mathrm{div})$-conforming velocity field \cite{MardalTaiWinther02,TaiWinther06,XieXuXue08,GuzmanNeilan12}, the conforming  methods \cite{XieXuXue08,JuntunenStenberg10}, 
the stabilized methods \cite{XieXuXue08,BadiaCodina09, JuntunenStenberg10}, the $H(\mathrm{div})$-conforming discontinuous Galerkin method \cite{KonnoStenberg11}, and the hybridized $H(\mathrm{div})$-conforming discontinuous Galerkin method \cite{KonnoStenberg12},
and for other alternative formulations, including the vorticity-velocity-pressure formulation \cite{VassilevskiVilla14, AGMR15}, the pseudostress-based formulation \cite{GaticaGaticaSequeira15}, and a dual-mixed formulation  \cite{HowellNeilanWalkington16}.

In this paper, we propose and study a class of  high-order,
parameter-free, $H(\mathrm{div})$-conforming HDG method for  the Brinkman equations \eqref{nb} on both simplicial and rectangular meshes. This is the first HDG method for the Brinkman equations based on a velocity gradient-velocity-pressure formulation.
Our method can be considered as a natural, stable extension to the Brinkman equations of the high-order,  parameter-free, $H(\mathrm{div})$-conforming HDG method for the Stokes problem on simplicial meshes \cite{CockburnSayasHDGStokes14}.
%
Three distinctive properties of the method make it attractive.
{Firstly, our method provides optimal error estimate in $L^2$-norms for the velocity that is 
robust with respect to viscosity/permeability ratio $\nu/\gamma$ (Theorem~\ref{thm:ns-error}, Corollary~\ref{coro:ns-error}), 
and superconvergent error estimate in the $L^2$-norm 
of one order higher for a suitable projection of the velocity error (under a regularity assumption on the dual problem). 
To the best of our knowledge, this is the first superconvergent velocity estimate for the Brinkman equations.
Secondly, thanks to $H(\mathrm{div})$-conformity of the velocity, our velocity error estimates are independent of 
the pressure regularity (see Corollary~\ref{coro:ns-error} and Theorem~\ref{thm:ns-superror}). 
Such pressure-robustness property is highly appreciated for incompressible flow problems \cite{Link14,LinkeMerdon16}. 
Finally, our error analysis, which is quite different from and more straightforward than that in \cite{CockburnSayasHDGStokes14} for the Stokes flow, is based on a so-called discrete $H^1$-stability result (see Theorem~\ref{thm:dh1}), which is the essential ingredient in the  analysis of
velocity gradient-velocity-pressure HDG formulation of 
the incompressible Navier-Stokes equations.
}
We specifically remark that no stabilization parameter enters in our method, which has to be  compared with the hybridized $H(\mathrm{div})$-conforming discontinuous Galerkin method \cite{KonnoStenberg12} in the classical velocity-pressure formulation, where Nitsche's penalty method is used to impose tangential continuity of the velocity field and the stabilization parameter needs to be ``sufficiently large''.  

The organization of the paper is as follows. In Section 2, we introduce the parameter-free $H(\mathrm{div})$-conforming HDG 
method and give the main results on a priori error estimates. In Section 3, we prove our main results in Section 2. 
In Section 4, we discuss the hybridization of the $H(\mathrm{div})$-conforming HDG method. {In Section 5, we provide 
preliminary two-dimensional numerical experiments on triangular and rectangular meshes to validate our theoretical results}. 
We end in Section 6 with some concluding remarks.

\section{Main results: Superconvergent $H(\mathrm{div})$-conforming HDG}
\label{sec:main}
{In this section, we first introduce the notation that will be used throughout the paper, and then present the finite element spaces that define the $H(\mathrm{div})$-conforming HDG methods. We conclude with an a priori error estimates along with a key 
inequality that we call {\it discrete $H^1$-stability}.
}
\subsection{Meshes and trace operators}
We denote by $\Oh:=\{K\}$ (the mesh) a shape-regular conforming triangulation of the domain $\Omega\subset \bR^d$
into {\it affine-mapped} simplices (triangles if $d=2$, tetrahedron if $d=3$) or hypercubes (squares if $d=2$, cubes if $d=3$),
and by $\Eh$ (the mesh skeleton) the set of facets $F$ (edges if $d=2$, faces if $d=3$) of the elements $K \in \Oh$.
Let $\mathcal{F}(K)$ denote the set of facets $F$ of the element K.
We set $h_F := \mathrm{diam}(F), h_K := \mathrm{diam}(K)$ and
$h := \max_{K\in\Oh}h_K$.

\newcommand{\Kbar}{\underline{\mathsf{K}}}
\newcommand{\Fbar}{\underline{\mathsf{F}}}

Let $\Kbar$ be the reference element ($d$-dimensional simplex or hypercube), and $\Fbar$ be the reference facet 
($d-1$-dimensional simplex or hypercube).
We denote
$\Phi_K: \Kbar\rightarrow K$ and $\Phi_F: \Fbar\rightarrow F$ as the associated affine mappings.

For a $d$-dimensional vector-valued function $\vv$ on an element  $K\subset \bR^d$ with sufficient regularity,
we denote
by
\begin{align}
\label{trace}
\tr_t^F(\vv):= \left.\left(\vv-(\vv\cdot\n_F)\,\n_F)\right)\right|_F
\hspace{0.5cm}\text{ and }\hspace{0.5cm}\tr_n^F(\vv):= \left.\left(\vv\cdot\n_F\right)\n_F \right|_F
\end{align}
the tangential and normal traces of $\vv$ on the facet $F\in \mathcal{F}(K)$, where $\n_{F}$ is
the unit normal vector to $F$. Note that the above trace operators are independent of the direction of
the normal $\n_F$. Whenever there is no confusion, we suppress the superscript and denote
$\tr_t(\vv)$ and $\tr_n(\vv)$ as the related tangential and normal traces, respectively.
With an abuse of notation, we also denote
\begin{align*}
\tr_t(\vvhat):= \left.\left(\vvhat-(\vvhat\cdot\n_F)\,\n_F)\right)\right|_F
\hspace{0.5cm}\text{ and }\hspace{0.5cm}\tr_n(\vvhat):= \left.\left(\vvhat\cdot\n_F\right)\n_F \right|_F
\end{align*}
for a $d$-dimensional vector-valued function $\vv$ on a facet  $F\subset \bR^{d-1}$ with sufficient regularity.

\subsection{The finite element spaces}
Now, we define the finite element spaces associated with the mesh $\Oh$ and mesh skeleton $\Eh$
via appropriate mappings (cf. \cite{BrennerScott08}) from (polynomial) spaces on the reference elements.

\newcommand{\vbar}{\underline{\bld{\mathsf{ v}}}}
\newcommand{\qbar}{\underline{{\mathsf{ q}}}}
\newcommand{\vhatbar}{\underline{\widehat{\bld{\mathsf{v}}}}}

We use the following mapped finite element spaces on the mapped element $K$ and facet $F$:
\begin{subequations}
\label{hdg-space}
\begin{alignat}{3}
\label{hdg-space-g-row}
\GG^\mathrm{row}(K):=&\;\{\vv\in{L}^2(K)^{d}:\;&&\;\vv=
\frac{1}{\mathrm{det}\,\Phi_K'}\Phi_K'\,\vbar\circ\Phi_K^{-1},\;\;\vbar\in\GG^\mathrm{row}(\Kbar)\},
\\
\label{hdg-space-v}
\VV(K):=&\;\{\vv\in{L}^2(K)^d:&&\;
\vv=
\frac{1}{\mathrm{det}\,\Phi_K'}\Phi_K'\,\vbar\circ\Phi_K^{-1},\;\;\vbar\in\VV(\Kbar)\},
\\
\label{hdg-space-q}
Q(K):=&\;\{\qq\in{L}^2(K):&&\;\qq=\qbar\circ\Phi_K^{-1},\;\;\qbar\in \QQ(\Kbar)\},
\\
\label{hdg-space-m}
\MM(F):=&\;\{\vvhat\in{L}^2(F)^d:&&\;\vvhat=\vhatbar\circ\Phi_F^{-1},\;\;\vhatbar\in \MM(\Fbar)\}.
\end{alignat}
\end{subequations}
Here $\Phi_K$ and $\Phi_F$ are the affine mappings introduced above, and $\Phi_K'$ is the Jacobian matrix of the mapping $\Phi_K$.
Note that the vector spaces in \eqref{hdg-space-g-row} and \eqref{hdg-space-v}
are obtained from the well-known {\it Piola transformation} which preserve normal continuity (cf. \cite{Duran08}).

The polynomial spaces on the reference elements are given in Table \ref{table-example-m}.
\begin{table}[ht]
\caption{The reference finite element spaces}
 \centering
\begin{tabular}{l c c c c}
\hline
\noalign{\smallskip}
element & $\GG^\mathrm{row}(\Kbar)$ & $\VV(\Kbar)$ &
 $\QQ(\Kbar)$ & $\MM(\Fbar)$ \\
\noalign{\smallskip}
\hline
\noalign{\smallskip}
simplex & $\pol_k(\Kbar)^d$
& $\RT_k(\Kbar)$
& $\pol_k(\Kbar)$ & $\pol_k(\Fbar)^d$\\
\noalign{\smallskip}
\hline
\noalign{\smallskip}
hypercube &
$\BDM_k(\Kbar)$
&
$\BDFM_k(\Kbar)$
& $\pol_k(\Kbar)$& $\pol_k(\Fbar)^d$\\
\noalign{\smallskip}
\hline
\end{tabular}
\label{table-example-m}
\end{table}

Here we denote $\pol_k(D)$ and $\widetilde{\pol}_k(D)$  as the polynomials of degree no
greater than $k$, and homogeneous polynomials of degree $k$, respectively, on the domain $D$.
The vector space $\RT_k(\Kbar)$
on the reference simplex is the following Raviart-Thomas-Ned\'el\'ec space, see
\cite{RaviartThomas77,Nedelec80},
\[
 \RT_k(\Kbar):=\pol_k(\Kbar)^d\oplus \bld x\,\widetilde{\pol}_k(\Kbar),
\]
the vector space $\BDM_k(\Kbar)$ on the reference hypercube is the following Brezzi-Douglas-Marini space,
see \cite{BrezziDouglasMarini85,BrezziDouglasDuranFortin87,ArnoldAwanou14},
\begin{align*}
  \BDM_k(\Kbar):=\left\{\begin{tabular}{l l}
              $\pol_k(\Kbar)^d\oplus \curls\{x\,y^{k+1},y\,x^{k+1}\}$ &
              if $d=2$,
                            \vspace{0.3cm}
\\
              $\pol_k(\Kbar)^d\oplus
\curls\left\{
 \begin{tabular}{c}
   $x\,\widetilde{\pol}_k(y,z)(y\grads z-z\grads y),$\\
   $y\,\widetilde{\pol}_k(z,x)(z\grads x-x\grads z),$\\
 $z\,\widetilde{\pol}_k(x,y)(x\grads y-y\grads x)$\\
 \end{tabular}
\right\}$& if $d=3$,
              \end{tabular}\right.
\end{align*}
and the vector space $\BDFM_k(\Kbar)$ on the reference hypercube is the following Brezzi-Douglas-Fortin-Marini space,
see \cite{BrezziDouglasFortinMarini87},
\begin{align*}
  \BDFM_k(\Kbar):=\left\{\begin{tabular}{l l}
              $\pol_k(\Kbar)^d\oplus \left[
              \begin{tabular}{c}
               $x\,\widetilde\pol_k(\Kbar)$
               \vspace{0.2cm}\\
               $y\,\widetilde\pol_k(\Kbar)$
              \end{tabular}
              \right]$ &
              if $d=2$,
                            \vspace{0.3cm}
\\
              $\pol_k(\Kbar)^d\oplus
              \left[
              \begin{tabular}{c}
               $x\,\widetilde\pol_k(\Kbar)$
               \vspace{0.2cm}\\
               $y\,\widetilde\pol_k(\Kbar)$
               \vspace{0.2cm}\\
               $z\,\widetilde\pol_k(\Kbar)$
               \end{tabular}
              \right]
$& if $d=3$.
              \end{tabular}\right.
\end{align*}

Next, for the vector-valued finite element space $\GG^\mathrm{row}(K)$ given
in \eqref{hdg-space-g-row}, we denote
\begin{align}
\label{gg-space}
 \GG(K):=
\left[\GG^\mathrm{row}(K)\right]^d
 \end{align}
as the tensor-valued space such that each of whose row is the space $\GG^\mathrm{row}(K)$.

We use the following finite element spaces on the mesh $\Oh$ and mesh skeleton $\Eh$
to define the $H(\mathrm{div})$-conforming HDG method in the next
section.
\begin{subequations}
\label{hdg-space-div}
\begin{alignat}{3}
\label{hdg-space-g}
\GG_h:= &\; \{\mg\in L^2(\Oh)^{d\times d}:\;&&\mg|_K\in \GG(K),\;\; K\in\Oh\}
\\
\label{hdg-space-v-0}
\VV_h:=&\;
\{\vv\in L^2(\Oh)^d:\;&&
\vv|_K \in \VV(K),\;\; K\in\Oh\},
\\
\label{hdg-space-v-div}
\VVdiv:=&\;
\{\vv\in \VV_h:\;&&
\vv\in H(\mathrm{div};\Omega)\},
\\
\label{hdg-space-v-div-0}
\VVdiv(0):=&\;
\{\vv\in \VVdiv:\;&&
\mathrm{tr}_n(\vv)|_{\partial \Omega}=0
\},
\\
\label{hdg-space-q-1}
{Q_h}:=&\;
\{q\in L^2(\Oh):\;&&q|_K\in\QQ(K),\;\; K\in\Oh\},
\\
\label{hdg-space-q-0}
\mathring{Q_h}:=&\;
\{q\in Q_h:\;&&
\bint{q}{1} = 0
\},
\\
\MMh:=&\;\{\vvhat\in L^2(\Eh)^d:\;&&\vvhat|_F\in \MM(F),\;\; F\in\Eh\},
\\
\MMh(0):=&\;\{\vvhat\in \MMh:\;&&\vvhat|_F = \boldsymbol{0},\;\; F\in\Eh\},
\\
\label{hdg-space-m-t}
\Mth:=&\;\{\vvhat\in \MMh:\;&&\mathrm{tr}_n(\vvhat)|_F=0,\;\; F\in\Eh\},
\\
\label{hdg-space-m-t-0}
\Mth(0):=&\;\{\vvhat\in \Mth:\;&&\mathrm{tr}_t(\vvhat)|_{\partial \Omega}=\boldsymbol{0}\}.
\end{alignat}
\end{subequations}

\newcommand{\trn}{\mathrm{tr}_n}
\newcommand{\trt}{\mathrm{tr}_t}
\subsection{The $H(\mathrm{div})$-conforming HDG method}
Now, we are ready to present the $H(\mathrm{div})$-conforming HDG method for the Brinkman equations \eqref{nb}.

It is defined as the unique element $(\ml^h,\boldsymbol{u}^h,p^h,{\uthat}^h)\in \GG_h\times \VVdiv(0)\times \Qh\times
\Mth(0)$
such that the following weak formulation holds:
\begin{subequations}
\label{Hdiv-HDG-equations}
 \begin{alignat}{3}
 \label{Hdiv-HDG-equations-1}
 \bint{\ml^h}{\nu\,\mg^h}-\bint{\gradv \bld u^h}{\nu\,\mg^h}   + \bintEh{\trt(\bld u^h)-\uthat^h}{ \trt(\nu\,\mg^h\, \n)} & = 0, \\
 \label{Hdiv-HDG-equations-2}
 \bint{\nu\,\ml^h}{\gradv \vv^h} - \bintEh{\trt(\nu\,\ml^h\,\n)}{\trt(\vv^h)-\vthat^h}&
 \\
- \bint{p^h}{\divs \vv^h}+\bint{\gamma\, \bld u^h}{\vv^h}&=(\bld f,\vv^h)_\Oh,\nonumber \\
 \label{Hdiv-HDG-equations-3}
 \bint{\divs\bld u^h}{q^h} & = (g, q^h)_\Oh,
\end{alignat}
\end{subequations}
for all  $(\mg^h,\boldsymbol{v}^h,q^h,{\vthat}^h)\in \GG_h\times \VVdiv(0)\times \Qh\times
\Mth(0)$.
Here we write $\inp{\eta}{\zeta} := \sum_{K \in \Oh} (\eta, \zeta)_K,$
where $(\eta,\zeta)_K$ denotes the integral of $\eta\zeta$ over the domain $K \subset \mathbb{R}^n$. We also write
$\bintEh{\eta}{\zeta}:= \sum_{K \in \Oh}\langle \eta \,,\,\zeta \rangle_{\dK}$, where
$\langle \eta \,,\,\zeta \rangle_{\dK}:=\sum_{F \in \mathcal{F}(K)}  \langle \eta \,,\,\zeta \rangle_{F},$
and $\langle \eta \,,\,\zeta \rangle_{F}$ denotes the integral of $\eta \zeta$ over the facet $F \subset \mathbb{R}^{n-1}$
and where $\partial \Oh := \{ \partial K: K \subset \Oh \}$. When vector-valued or tensor-valued functions are involved, we use similar notation.

As mentioned in the Introduction, we postpone to Section \ref{sec:hybridization}
to discuss the efficient implementation of the above method via hybridization. 
Here we focus on the presentation of its (superconvergent) a priori error estimates.

\subsubsection{Discrete $H^1$-stability}
\label{subsec:dh1}
We first obtain a key result, which will be used to  prove the error estimates presented in the next subsection,
on the control of a {\em discrete $H^1$-norm} of the pair $(\bld u^h, \uthat^h)\in \VVdiv\times \Mth$ by 
the $L^2$-norm of a tensor field.

For a pair $(\bld v^h, \vthat^h)\in \VVdiv\times \Mth$, we denote
its discrete $H^1$-norm as follows:
\begin{align}
 \label{discrete-H1-norm}
 \vertiii{(\bld u^h, \uthat^h)}_{1,\Oh} :=\left(
\sum_{K\in\Oh} \|\gradv \bld u^h\|_K^2 + \sum_{F\in\Eh}h_F^{-1}\|
 \trt(\bld u^h)-\uthat^h\|_F^2
 \right)^{1/2}
\end{align}

\begin{theorem}[Discrete $H^1$-stability]
\label{thm:dh1}
Let $(\mr, \bld z^h, \zthat^h)\in L^2(\Oh)^{d\times d}\times \VVdiv\times \Mth$
satisfy the following equation
\begin{align}
\label{dh1-equation}
 \bint{\mr}{\mg^h}-\bint{\gradv \bld z^h}{\mg^h}
 + \bintEh{\trt(\bld z^h)-\zthat^h}{ \trt(\mg^h\, \n)}  = 0
\end{align}
for all $\mg^h\in \GG_h$, then we have
\begin{align}
 \label{discrete-H1-control}
 \vertiii{\left(\bld z^h, \zthat^h\right)}_{1,\Oh}\le C\,\|\mr\|_\Oh,
\end{align}
with a constant $C$ depends only on the polynomial degree $k$ and the shape-regularity of the elements
$K\in\Oh$.
\end{theorem}

\subsubsection{A priori error estimates}
\label{subsec:error}
We are now ready to present the a priori error estimates for the method \eqref{Hdiv-HDG-equations}.
We compare the numerical solution against suitably chosen projections.

\subsubsection*{The projections}
In the following, we denote $P_\GG$, $P_\VV$, $\Piq$, $P_{\bld{M}^t}$ to be the $L^2$-projections onto
$\GG_h$, $\Vh$, $\Qh$, and $\Mth$ respectively.
Moreover, we set
\begin{alignat*}{3}
 \egg = \Pigg \ml - \ml^h,\;\;&  \euu = \Piww \bld u - \bld u^h, &&\;\;
 \epp = \Piq p - p^h, && \;\;\euuhat = \Pimm \bld u - \uthat^h,\\
 \dgg = \ml -\Pigg \ml, \;\;& \duu = \bld u - \Piww \bld u, &&\;\;
 \dpp = p -\Piq p, && \;\;\duuhat =\trt(\bld u) -  \Pimm \bld u.
\end{alignat*}
Here the projection $\Piww\bld u\in \Vh$ whose restriction to an element $K$ is the unique function in
$\VV(K)$ such that
\begin{subequations}
 \label{bu-projection}
\begin{alignat}{2}
 \label{bu-projection-1}
 (\Piww \bld u, \vv)_K = &\; (\bld u,\vv)_K &&\;\;\forall \;\vv\in \divv \GG(K),\\
 \label{bu-projection-2}
 \bintF{\trn(\Piww \bld u)}{\trn(\vvhat)} = &\; \bintF{\trn(\bld u)}{\trn(\vvhat)} &&\;\;\forall\; \vvhat\in \MM(F),\;\;
 \forall F\in \mathcal{F}(K).
\end{alignat}
\end{subequations}
Recall that the spaces $\VV(K)$, $\MM(F)$, and $\GG(K)$ are defined in \eqref{hdg-space}, and
\eqref{gg-space}, respectively.

When $K$ is a simplex, the above projection is nothing but the Raviar-Thomas projection,
see \cite{RaviartThomas77,Nedelec80}; when $K$ is a hypercube, the above projection is nothing but the
Brezzi-Douglas-Fortin-Marini projection, see \cite{BrezziDouglasFortinMarini87}.

The following approximation property of the above projection is well-known; see \cite[Chapter 2]{BoffiBrezziFortin13}.

\begin{lemma}
\label{lemma:projection-b}
There exists a unique function $\Piww \bld u\in\VVdiv$ defined element-wise by the equations \eqref{bu-projection}.
Moreover, there exists a constant $C$ only depending on the polynomial degree and shape-regularity of the elements $K\in\Oh$ such that
   \begin{alignat}{2}
   \label{projection-approx-1-b}
    \|\Piww \bld u - \bld u\|_\Oh\le &\;C\, \left(\|P_{\VV} \bld u - \bld u\|_\Oh
    +\sum_{K\in\Oh}h_K^{1/2}\|P_{\VV} \bld u -\bld  u\|_\dK \right).
 \end{alignat}
\end{lemma}

\subsubsection*{The projection errors}
Now, we state our main results on the superconvergent error estimates.
\begin{theorem}
\label{thm:ns-error}
Let $(\ml^h,\bld u^h,p^h, \uthat^h)\in \GG_h\times \VVdiv(0)\times \Qh\times \Mth(0)$ be the numerical solution of
\eqref{Hdiv-HDG-equations},
then there exists a constant $C$, depending only on the polynomial degree $k$, the shape-regularity of the mesh $\Oh$, and the domain $\Omega$,  such that
\begin{subequations}
\begin{alignat}{2}
\label{error-u-l2-ns}
\|\euu\|_{\Oh} \le &\;C\,
\vertiii{(\euu,\euuhat)}_{1,\Oh},
\\
\label{error-u-h1-ns}
\vertiii{(\euu,\euuhat)}_{1,\Oh}
\le &\;
C\,\|\egg\|_{\Oh},\\
\label{error-q-ns}
\nu\|\egg\|_{\Oh}^2
+\|\gamma^{1/2}\,\euu\|_\Oh^2
\le &\;C\,\left(\sum_{F\in\Eh}\nu\,h_F\|\dgg\,\n\|_F^2+\|\gamma^{1/2}\,\duu\|_\Oh^2 \right).
\end{alignat}
\end{subequations}
\end{theorem}

Combing this result with Lemma \ref{lemma:projection-b}, we
immediately obtain optimal convergence of $L^2$-error for $\ml^h$ and $\bld u^h$, and
superconvergent discrete $H^1$-error for the pair $(\bld u^h,\uthat^h)$ comparing with the projection
$(\Piww \bld u, \Pimm \bld u)$; see the following corollary.
We omit the proof due to its simplicity. We specifically remark that the errors below are independent of the regularity of the pressure.

\begin{corollary}
\label{coro:ns-error}
Let $(\ml^h,\bld u^h,p^h, \uthat^h)\in \GG_h\times \VVdiv(0)\times \Qh\times \Mth(0)$ be the numerical solution of
\eqref{Hdiv-HDG-equations},
then there exists a constant $C$, depending only on the polynomial degree $k$, the shape-regularity of the mesh $\Oh$, and the domain $\Omega$,  such that
\begin{align*}
 \nu^{1/2}\left(\|\egg\|_{\Oh}
 +\vertiii{(\euu,\euuhat)}_{1,\Oh}\right)
 +
 \max\{ \nu^{1/2}\|\euu\|_\Oh, \|\gamma^{1/2}\,\euu\|_\Oh\} \le &\;C\,\Theta\, h^{k+1}
\end{align*}
where
\[
 \Theta:=
  \nu^{1/2}\,\|\ml\|_{k+1,\Omega}+
  \gamma_{\max}^{1/2}\,\|\bld u\|_{k+1,\Omega},
\]
and $\gamma_{\max}$ is the maximum eigenvalue of the inverse permeability tensor $\gamma$,
and $\|\cdot\|_{m,\Omega}$ denotes the $H^m$-norm on $\Omega$.
\end{corollary}

Next, we obtain optimal $L^2$-estimates for pressure for $k\ge 0$ and superconvergent
$L^2$-estimates for the projection error $\euu$ for $k\ge 1$ (with a $H^2$-regularity assumption for the dual problem).

{We assume that the following regularity estimate} holds
\begin{align}
\label{dual_prob_assum}
\|\Phi\|_{1,\Omega}+\|\boldsymbol{\phi}\|_{2,\Omega}+\|\varphi\|_{1,\Omega}\le C_r\|\boldsymbol{\theta}\|_{\Omega}
\end{align}
for the dual problem
\begin{subequations}
\label{dual_prob}
\begin{align}
\label{dual_prob_a}
\Phi-\gradv\boldsymbol{\phi}=0&\;\qquad \text{in}\quad\Omega,\\
\label{dual_prob_b}
-\nu\divv \Phi+\gamma\boldsymbol{\phi}-\nabla \varphi=\boldsymbol{\theta}&\;\qquad\text{in}\quad\Omega,\\
\label{dual_prob_c}
\divv\boldsymbol{\phi}=0&\;\qquad \text{in}\quad \Omega,\\
\label{dual_prob_d}
\boldsymbol{\phi}=0&\;\qquad \text{on}\quad \partial\Omega.
\end{align}
\end{subequations}
{
We notice that it is easy to see the dual problem (\ref{dual_prob}) is well-posed with 
$\Vert \boldsymbol{\phi} \Vert_{1,\Omega} \leq C \Vert \boldsymbol{\theta}\Vert_{\Omega}$. 
Obviously, $(\Phi, \boldsymbol{\phi}, \varphi)$ is the solution of the Stokes problem with 
the source term $\boldsymbol{\theta} - \gamma \boldsymbol{\phi}$. So, the regularity estimate (\ref{dual_prob_assum}) 
comes from that of the Stoke problem (see \cite{GiraultRaviart86}).
}

\begin{theorem}
\label{thm:ns-superror}
Let $(\ml^h,\bld u^h,p^h, \uthat^h)\in \GG_h\times \VVdiv(0)\times \Qh\times \Mth(0)$ be the numerical solution of
\eqref{Hdiv-HDG-equations},
then there exists a constant $C$, depending only on the polynomial degree $k$, the shape-regularity of the mesh $\Oh$, and the domain $\Omega$,  such that
\begin{align}\label{error_p}
\|\epp\|_{\Oh}\le C(\nu^{1/2}+\gamma_{\max}^{1/2})\,\Theta\,h^{k+1},
\end{align}
here $\gamma_{\max}$ and $\Theta$ are defined in Corollary \ref{coro:ns-error}.

In addition, if $k\ge 1$, the regularity assumption \eqref{dual_prob_assum} holds and
$\gamma\in W^{1,\infty}(\Omega)^{d\times d}$, then we have
\begin{align}\label{superror_u}
\|\euu\|_{\Oh}\le C\,C_r\,\left((\nu^{1/2}+\gamma_{\max}^{1/2})\,\Theta+
\|\gamma\|_{1,\infty}\|\bld u\|_{k+1}\right)h^{k+2}.
\end{align}

\end{theorem}

\section{Proof of Theorem \ref{thm:dh1}, Theorem \ref{thm:ns-error} and Theorem \ref{thm:ns-superror}}
\label{sec:proof}
In this section, we prove the main results in Section \ref{sec:main}, namely, Theorem \ref{thm:dh1}, Theorem \ref{thm:ns-error} and Theorem \ref{thm:ns-superror}.

The following result is a key ingredient to prove Theorem \ref{thm:dh1}. We postpone its proof to Appendix.
\begin{lemma}
\label{lemma:key}
 Given $(\bld z^h, \zzhat^h)\in \VV(K)\times \MM(\dK)$ where
 \[
  \MM(\dK):=\{\vvhat\in L^2(\dK)^d:\;\vvhat|_F\in \MM(F)\;\;\forall F\in\mathcal{F}(K)\},
 \]
there exists a unique function $\mr^h\in \GG(K)$ such that
\begin{subequations}
\label{g-proj}
\begin{alignat}{2}
\label{g-proj-1}
 ({\mr^h},{\mg^h})_K = &\; ({\gradv \bld z^h},{\mg^h})_K&&\;\;\forall \mg^h\in \gradv\VV(K)\oplus \GG_{\mathrm{sbb}}(K),\\
\label{g-proj-2}
 \bintK{\trt(\mr^h\,\n)}{\trt(\vvhat)} = &\; \bintK{\trt(\zzhat^h)}{\trt(\vvhat)}&&\;\;\forall \vvhat^h\in \MM(\dK),
\end{alignat}
\end{subequations}
where
\[
 \GG_{\mathrm{sbb}}(K):=\{\mg\in \GG(K):\;\;\divv \mg = 0,\;\;\trn^F(\mg\,\n)=0 \;\forall F\in\mathcal{F}(K)\}.
\]
Moreover, there exists a constant $C$ only depending on the shape-regularity of the element $K$ such that
\begin{align}
 \label{estimate-1}
 \|\mr^h\|_K\le C\left(\|\gradv \bld z^h\|_K^2 + \sum_{F\in\mathcal{F}(K)}h_F\|\trt(\zzhat^h)\|_F^2\right)^{1/2}
\end{align}
\end{lemma}

Now, we are ready to prove Theorem \ref{thm:dh1}.
\subsection*{Proof of Theorem \ref{thm:dh1}}
\begin{proof}
By Lemma \ref{lemma:key}, for any $\bld z^h\in \VV(K)$ and $\zthat^h\in \{\vvhat\in\MM(\dK):\; \trn(\vvhat)=0\}$,
there exists $\mg^h\in \GG(K)$  such that
\begin{align*}
 (\gradv \bld z^h, \mg^h)_K - \bintK{\trt(\bld z^h)-\zthat^h}{\trt(\mg^h\,\n)}
 =&\;\|\gradv \bld z^h\|_K^2\\
 &\;+\sum_{F\in\mathcal{F}(K)}h_F^{-1}\|\Pimm(\trt(\bld z^h))-\zthat^h\|_F^2
\end{align*}
and $\|\mg^h\|_K\le C\, (\|\gradv \bld z^h\|_K^2+\sum_{F\in\mathcal{F}(K)}h_F^{-1}\|\Pimm(\trt(\bld z^h))-\zthat^h\|_F^2)^{1/2}$.
Taking such $\mg^h$ in \eqref{dh1-equation}, we get
\begin{align*}
\|\gradv \bld z^h\|_K^2+&\;\sum_{F\in\mathcal{F}(K)}h_F^{-1}\|\Pimm(\trt(\bld z^h))-\zthat^h\|_F^2
 =(\mr, \mg^h)_K\\
 \le&\; C\, \left(\|\gradv \bld z^h\|_K^2+\sum_{F\in\mathcal{F}(K)}h_F^{-1}\|\Pimm(\trt(\bld z^h))-\zthat^h\|_F^2\right)^{1/2}\,\|\mr\|_K.
\end{align*}
Hence,
\begin{align}
\label{key-2}
 \left(\|\gradv \bld z^h\|_K^2+\sum_{F\in\mathcal{F}(K)}h_F^{-1}\|\Pimm(\trt(\bld z^h))-\zthat^h\|_F^2\right)^{1/2}\le C\,\|\mr\|_K.
\end{align}

Moreover, on each facet $F\in\mathcal{F}(K)$, we have
\[
\| \trt(\bld z^h) -\Pimm(\trt(\bld z^h))\|_{F}
= \| \bld z^h -P_{\bld M}(\bld z^h)\|_{F}
\le \,\| \bld z^h -\overline{\bld z^h}\,\|_{F}
\le C\,h_K^{1/2}\| \gradv \bld z^h\,\|_{K},
\]
where $\overline{\bld z^h}$ is the average of $\bld z^h$ in the element $K$ and the last inequality is the
Poinc\'are inequality.
Combining the above result with \eqref{key-2}, we obtain
\[
  \vertiii{\left(\bld z^h, \zthat^h\right)}_{1,K}\le C\,\|\mr\|_K.
\]
The proof of Theorem \ref{thm:dh1} is completed by summing the above estimate over all the elements $K\in\Oh$.
\end{proof}

We use the following error equation to prove Theorem \ref{thm:ns-error}.
To simplify notation, we denote
\begin{align}
\label{bilinear}
 {B}_h(\ml, \bld u, p,\uthat; \mg, \bld v, q,\vthat):=&\;
 \bint{\ml}{\nu\,\mg}-\bint{\gradv \bld u}{\nu\,\mg} \\
&\; + \bintEh{\trt(\bld u)-\uthat}{ \trt(\nu\,\mg\, \n)}  \nonumber\\
&\; + \bint{\nu\,\ml}{\gradv \vv} - \bintEh{\trt(\nu\,\ml\,\n)}{\trt(\vv)-\vthat}\nonumber\\
&\;- \bint{p}{\divs \vv}+\bint{\gamma\, \bld u}{\vv}\nonumber\\
&\;
+ \bint{\divs\bld u}{q}.\nonumber
\end{align}

\begin{lemma}
 Let $(\ml, \bld u, p)$ be the solution to \eqref{nb}, and
 $(\ml^h, \bld u^h, p^h,\uthat^h)$ be the numerical solution to \eqref{Hdiv-HDG-equations}.
 Then, we have
 \begin{align}
\label{error-equation}
 {B}_h(\egg, \euu, \epp,\euuhat; \mg^h, \bld v^h, q^h,\vthat^h)
 =&\; \bintEh{\trt(\nu\,\dgg\,\n)}{\trt(\vv^h)-\vthat^h}\\
& \;-\bint{\gamma\,\duu}{\vv^h}.\nonumber
 \end{align}
for all  $(\mg^h,\boldsymbol{v}^h,q^h,{\vthat}^h)\in \GG_h\times \VVdiv(0)\times \Qh\times
\Mth(0)$.
\end{lemma}
\begin{proof}
By \eqref{nb}, \eqref{Hdiv-HDG-equations}, and \eqref{bilinear}, we have
\begin{align*}
  {B}_h(\ml^h, \bld u^h, p^h,\uthat^h; \mg^h, \bld v^h, q^h,\vthat^h)
  =(\bld f,\vv^h)_\Oh + (g,q^h)_\Oh\\
  {B}_h(\ml, \bld u, p,\trt(\bld u); \mg^h, \bld v^h, q^h,\vthat^h)
  =(\bld f,\vv^h)_\Oh + (g,q^h)_\Oh
\end{align*}
for all
$(\mg^h,\boldsymbol{v}^h,q^h,{\vthat}^h)\in \GG_h\times \VVdiv(0)\times \Qh\times
\Mth(0)$.
Hence,
\begin{align*}
 {B}_h(\egg, \euu, \epp,\euuhat; \mg^h, \bld v^h, q^h,\vthat^h)
  =
-{B}_h(\dgg, \duu, \dpp,\duuhat; \mg^h, \bld v^h, q^h,\vthat^h).
\end{align*}
Using orthogonality properties of the projections, we easily obtain
\[
 {B}_h(\dgg, \duu, \dpp,\duuhat; \mg^h, \bld v^h, q^h,\vthat^h)
 = -\bintEh{\trt(\nu\,\dgg\,\n)}{\trt(\vv^h)-\vthat^h}+\bint{\gamma\,\duu}{\vv^h}.
\]
This completes the proof.
\end{proof}

Now, we are ready to prove Theorem \ref{thm:ns-error}.
\subsection*{Proof of Theorem \ref{thm:ns-error}}
\begin{proof}
By \cite[Theorem 2.1]{DiPietroDroniouErn10}, we have
\[
 \|\euu\|_\Oh \le C\, \left(
 \|\gradv \euu\|_\Oh
 +
\sum_{F\in\mathcal{F}(K)}h_F^{-1} \left\|\jmp{\euu}\right\|_F^2
 \right)^{1/2}.
\]
Here $\jmp{\euu}:=\euu^+-\euu^-$ denotes the jump of $\euu\in \VVdiv(0)$ on a interior facet $F:=K^+\cap K^-$, and
$\jmp{\euu}:=\euu$ on a boundary facet $F\subset \partial\Omega$,
where $\euu^{\pm} = \euu|_{K^\pm}$.
Since $\euu$ is $H(\mathrm{div})$-conforming and has vanishing normal trace on the boundary,
we have $\trn(\jmp{\euu})=0$ for all facets $F\in \Eh$.
Hence, \[\jmp{\euu}=\trt(\jmp{\euu}).\]
By triangle inequality, we have
\[
\|\trt(\jmp{\euu})\|_F\le
\|\trt(\euu^+)-\euuhat\|_F
+\|\trt(\euu^-)-\euuhat\|_F.
\]
Combing the above estimates, we finish the proof of the first error estimate \eqref{error-u-l2-ns}.

The second error estimate \eqref{error-u-h1-ns} comes directly from Theorem \ref{thm:dh1}.

Now, let us prove the last error estimate \eqref{error-q-ns}.
Taking $(\mg^h, \bld v^h, q^h,\vthat^h):=(\egg, \euu, \epp,\euuhat)$, we obtain
 \begin{align*}
    \nu\|\egg\|_\Oh^2+ \|\gamma^{1/2}\euu\|_\Oh^2
  = &\; -\bintEh{\trt(\nu\,\dgg\,\n)}{\trt(\euu)-\euuhat}+\bint{\gamma\,\duu}{\euu}\\
  \le  &\; \sum_{F\in\Eh}\left(h_F^{1/2}\|\trt(\nu\,\dgg\,\n)\|_F\,h_F^{-1/2}\|{\trt(\euu)-\euuhat}\|_F
 \right) \\
 &\; +\|\gamma^{1/2}\duu\|_\Oh\|\gamma^{1/2}\euu\|_\Oh\\
\;\;\;\;\;  \le
\;C\,
\Big(\sum_{F\in\Eh}\nu\,h_F&\|\dgg\,\n\|_F^2+\|\gamma^{1/2}\,\duu\|_\Oh^2 \Big)^{{1/2}}
\,(    \nu\|\egg\|_\Oh^2+ \|\gamma^{1/2}\euu\|_\Oh^2 )^{1/2}
 \end{align*}
 by the Cauchy-Schwartz inequality.

This completes the proof of Theorem \ref{thm:ns-error}.
\end{proof}

The following result is used to prove the velocity estimate in
Theorem \ref{thm:ns-superror}.
\begin{lemma}
\label{lemma:dual}
Let $(\Phi,\boldsymbol{\phi},\varphi)$ be the solution to the dual problem \eqref{dual_prob}
for $\bld{\theta}\in L^2(\Oh)^d$. We have
\begin{align}\label{supererror_identity}
\bint{\euu}{\bld \theta}=&\;\bintEh{\nu\,\egg\,\n}{\delphi}+\bintEh{\trt(\nu\,\dgg\,\n)+\trt(\nu\,\egg\,\n)}{\Piww\bld\phi-\Pim\bld\phi}\nonumber\\
&\;+\bintEh{\trt(\euu)-\euuhat}{\nu\,\delPhi\n}+\bint{\gamma\,\euu}{\delphi}-\bint{\gamma\,\duu}{\Piww\bld\phi}\nonumber\\
=:&\;T_1+T_2+T_3+T_4+T_5,
\end{align}
where $\delPhi=\Phi-\Pigg\Phi,~\delphi=\bld \phi-\Piww\bld\phi,~\delvphi=\varphi-\Piq\varphi$.
\end{lemma}
\begin{proof}
By \eqref{dual_prob_a}-\eqref{dual_prob_c}, we have
\begin{align*}
\bint{\euu}{\bld \theta}=&\;-\bint{\euu}{\nu\,\divv\Phi}+\bint{\euu}{\nu\bld\phi}-\bint{\euu}{\grads\varphi}\\
&\;-\bint{\nu\egg}{\Phi}+\bint{\nu\,\egg}{\gradv\bld\phi}-\bint{\epp}{\divv\bld\phi}\\
=&\;-\bint{\euu}{\nu\,\divv\Pigg\Phi}-\bint{\euu}{\nu\,\divv\delPhi}-\bint{\euu}{\grads\Piq\varphi}-\bint{\euu}{\grads\delvphi}\\
&\;+\bint{\euu}{\gamma\bld\phi}-\bint{\nu\,\egg}{\Pigg\Phi}+\bint{\nu\,\egg}{\gradv\bld\phi}-\bint{\epp}{\divv\bld\phi}.
\end{align*}
Taking $(\mg^h, \bld v^h, q^h,\vthat^h):=(\Pigg\Phi, \bld 0, -\Piq\varphi,0)$ in the error equation \eqref{error-equation}, putting
the result identity into the above expression and simplifying, we have
\begin{align*}
\bint{\euu}{\bld \theta}=&\;-\bintEh{\euu}{\nu\,\Pigg\Phi\n}-\bintEh{\euu}{\Piq\varphi\n}\\
&\;+\bintEh{\trt(\euu)-\euuhat}{\trt(\nu\,\Pigg\Phi\n)}-\bint{\euu}{\nu\,\divv\delPhi}-\bint{\euu}{\grads\delvphi}\\
&\;+\bint{\euu}{\gamma\bld\phi}+\bint{\nu\,\egg}{\gradv\bld\phi}-\bint{\epp}{\divv\bld\phi}\\
=&\;-\bintEh{\euu}{\nu\,\Pigg\Phi\n}-\bintEh{\euu}{\Piq\varphi\n}\\
&\;+\bintEh{\trt(\euu)-\euuhat}{\trt(\nu\,\Pigg\Phi\n)}-\bintEh{\euu}{\nu\,\delPhi\n}-\bintEh{\euu}{\delvphi\n}\\
&\;+\bint{\euu}{\gamma\bld\phi}+\bint{\nu\,\egg}{\gradv\bld\phi}-\bint{\epp}{\divv\bld\phi}\\
=&\;-\bintEh{\euu}{\nu\,\Phi\n}+\bintEh{\trt(\euu)-\euuhat}{\nu\,\Pigg\Phi\n}\\
&\;+\bint{\euu}{\gamma\bld\phi}+\bint{\nu\,\egg}{\gradv\bld\phi}-\bint{\epp}{\divv\bld\phi}\\
=&\;-\bintEh{\trt(\euu)-\euuhat}{\nu\,\delPhi\n}\\
&\;+\bint{\euu}{\gamma\bld\phi}+\bint{\nu\,\egg}{\gradv\bld\phi}-\bint{\epp}{\divv\bld\phi},
\end{align*}
by inserting the zero term $\bintEh{\euuhat}{\nu\,\Phi\n}$ and using the fact that $\bintEh{\euu}{\nu\,\Phi\n}=\bintEh{\trt(\euu)}{\nu\,\Phi\n}$
and $\bintEh{\euu}{\varphi\n}=0$.

Take $(\mg^h, \bld v^h, q^h,\vthat^h):=(0, \Piww\bld \phi, 0, \Pimm\bld\phi)$ in the error equation \eqref{error-equation}.
Denoting by $I:=\bint{\euu}{\gamma\bld\phi}+\bint{\nu\,\egg}{\gradv\bld\phi}-\bint{\epp}{\divv\bld\phi}$, we obtain,
\begin{align*}
I=&\;\bint{\euu}{\gamma\delphi}+\bint{\nu\,\egg}{\gradv\delphi}-\bint{\epp}{\divv\delphi}\\
&\;+\bint{\euu}{\gamma\Piww\bld\phi}+\bint{\nu\,\egg}{\gradv\Piww\bld\phi}-\bint{\epp}{\divv\Piww\bld\phi}\\
=&\;\bint{\euu}{\gamma\delphi}+\bint{\nu\,\egg}{\gradv\delphi}-\bint{\epp}{\divv\delphi}\\
&\;\bintEh{\trt(\nu\,\dgg\n)+\trt(\nu\,\egg\n)}{\trt(\Piww\bld\phi)-\Pimm\bld\phi}-\bint{\gamma\duu}{\Piww\bld\phi}\\
=&\;\bint{\euu}{\gamma\delphi}+\bintEh{\nu\,\egg\n}{\delphi}\\
&\;\bintEh{\trt(\nu\,\dgg\n)+\trt(\nu\,\egg\n)}{\Piww\bld\phi-\Pim\bld\phi}-\bint{\gamma\duu}{\Piww\bld\phi}.
\end{align*}
This completes the proof of Lemma \ref{lemma:dual}.
\end{proof}

Now we are ready to prove Theorem \ref{thm:ns-superror}.
\subsection*{Proof of Theorem \ref{thm:ns-superror}}
\begin{proof}
We first present the optimal error estimate for $\epp$ by applying an $inf$-$sup$ argument. It is well-known that the following $inf$-$sup$ condition holds for a positive constant $\kappa$,
(cf. \cite[Chapter 1, Corollary 2.4]{GiraultRaviart86}),
\begin{align}\label{inf-sup}
\sup_{\boldsymbol{\omega}\in H^1_0(\Omega)^d\backslash\{0\}}\frac{(\divv\boldsymbol{\omega},q)_\Omega}{\|\boldsymbol{\omega}\|_{1,\Omega}}\ge\kappa\|q\|_{\Omega}.
\end{align}
Here $\Vert \cdot \Vert_{1,\Omega}$ is the standard $H^{1}$-norm on $\Omega$. 

Since $\epp\in L^2_0(\Omega)$, we have by \eqref{inf-sup}
\begin{align}\label{inf-sup_ep}
\|e_p\|_{\Omega}\le\frac{1}{\kappa}\sup_{\boldsymbol{\omega}\in H^1_0(\Omega)^d\backslash\{0\}}\frac{(\divv\boldsymbol{\omega},e_p)_\Omega}{\|\boldsymbol{\omega}\|_{1,\Omega}}.
\end{align}
Taking $(\mg^h, \bld v^h, q^h,\vthat^h):=(0, \Piww\bld \omega, 0, \Pimm\bld\omega)$ in the error equation \eqref{error-equation} and applying the integration by parts, we can rewrite the numerator as follows:
\begin{align*}
\bint{\divv \bld\omega}{\epp}=&\;\bint{\divv\Piww\bld\omega}{\epp}+\bint{\divv(\bld\omega-\Piww\bld\omega)}{\epp}=\bint{\divv\Piww\bld\omega}{\epp}\\
=&\;\bint{\nu\,\egg}{\gradv\Piww\bld\omega}-\bintEh{\trt(\nu\,\egg\n)+\trt(\nu\,\dgg\n)}{\trt(\Piww\bld\omega)-\Pimm\bld\omega}\\
&\;+\bint{\gamma\euu}{\Piww\bld\omega}+\bint{\gamma\duu}{\Piww\bld\omega}\\
=&\;\bint{\nu\,\egg}{\gradv\Piww\bld\omega}-\bintEh{\trt(\nu\,\egg\n)+\trt(\nu\,\dgg\n)}{\Piww\bld\omega-\Pim\bld\omega}\\
&\;+\bint{\gamma\euu}{\Piww\bld\omega}+\bint{\gamma\duu}{\Piww\bld\omega}\\
=:&\;I_1+I_2+I_3+I_4.
\end{align*}
Then we will bound $I_1$-$I_4$ by Corollary \ref{coro:ns-error} as follows.
\begin{align*}
I_1\le&\;\nu\norm{\egg}\norm{\gradv\Piww\bld\omega}\le C\nu^{1/2}\Theta h^{k+1}\normh{\bld\omega}.\\
I_2\le&\;\nu(\normEh{\egg\n}+\normEh{\dgg\n})\normEh{\Piww\bld\omega-\Pim\bld\omega}\\
\le&\;C(\nu^{1/2}\Theta h^{k+1/2}+\nu\|\mathrm{L}\|_{k+1}h^{k+1/2})h^{1/2}\normh{\bld\omega}\le C\nu^{1/2}\Theta h^{k+1}\normh{\bld\omega}.\\
I_3\le&\;C\gamma^{1/2}_{\max}\norm{\gamma^{1/2}\euu}\norm{\Piww\bld\omega}\le C\gamma^{1/2}_{\max}\Theta h^{k+1}\normh{\bld\omega}.\\
I_4\le&\;C\gamma_{\max}\norm{\duu}\norm{\Piww\bld\omega}\le C\gamma_{\max}\|\bld u\|_{k+1} h^{k+1}\normh{\bld\omega}\\
\le&\;C\gamma^{1/2}_{\max}\Theta h^{k+1}\normh{\bld\omega}.
\end{align*}
Then we have
$$\bint{\divv \bld\omega}{\epp}\le C(\nu^{1/2}+\gamma^{1/2}_{\max})\Theta h^{k+1}\normh{\bld\omega}.$$
By \eqref{inf-sup_ep}, we obtain the estimate for $\epp$.

Now we give superconvergent estimate for $\euu$. By \eqref{supererror_identity}, it suffices to estimate the terms $T_1$ to $T_5$.
We apply Corollary \ref{coro:ns-error}, the regularity assumption \eqref{dual_prob_assum} and the Poinc\'{a}re inequality to bound these terms.
\begin{align*}
T_1\le&\;\nu\normEh{\egg\n}\normEh{\delphi}\le C\nu h^{-1/2}\norm{\egg}h^{3/2}\|\bld\phi\|_{2}\\
\le&\;C\nu^{1/2}\Theta h^{k+2}\norm{\bld\theta}.\\
T_2\le&\;\nu(\normEh{\dgg\n}+\normEh{\egg\n})\normEh{\Piww\bld\phi-\Pim\bld\phi}\\
\le&\;C(\nu\|\mathrm{L}\|_{k+1}h^{k+1/2}+\nu^{1/2}\Theta h^{k+1/2})h^{3/2}\|\bld\phi\|_2\le C\nu^{1/2}\Theta h^{k+2}\norm{\bld\theta}.\\
T_3\le&\;\nu h^{-1/2}\normEh{\trt(\euu)-\euuhat}h^{1/2}\normEh{\delPhi\n}\\
\le&\;C\nu^{1/2}\vertiii{(\euu,\euuhat)}_{1,\Oh}h\normh{\Phi}\le C\nu^{1/2}\Theta h^{k+2}\norm{\bld\theta}.\\
T_4\le&\;\gamma^{1/2}_{\max}\norm{\gamma^{1/2}\euu}\norm{\delphi}\le C\gamma^{1/2}_{\max}\Theta h^{k+2}\norm{\bld\theta}.\\
T_5=&\;\bint{(\gamma-P_{0,h}\gamma)\duu}{\Piww\bld\phi}+\bint{P_{0,h}\gamma\duu}{\Piww\bld\phi- \bar{\bld\phi}}\\
\le&\;\|\gamma-P_{0,h}\gamma\|_\infty\norm{\duu}\norm{\Piww\bld\phi}+|P_{0,h}\gamma|\norm{\duu}\norm{\Piww(\bld\phi-\bar{\bld\phi})}\\
\le&\;Ch\|\gamma\|_{1,\infty}h^{k+1}\|\bld u\|_{k+1}\|\bld\phi\|_2+C\|\gamma\|_{0,\infty}\,h^{k+1}\|\bld u\|_{k+1}h\norm{\gradv\bld\phi}\\
\le&\;C \|\gamma\|_{1,\infty}\|\bld u\|_{k+1}h^{k+2}\norm{\bld\theta},
\end{align*}
where $P_{0,h}$ is $L^2$ orthogonal projection onto $\pol_0(\Oh)^{d\times d}$ and $\bar{\bld\phi}$ is defined as
$$\bar{\bld \phi}=\frac{1}{|K|}\bintKK{\bld\phi}{1},\quad\forall K\in\Oh.$$
Combining all the above estimates, we have
$$\|\euu\|_{\Oh}\le C(\nu^{1/2}\Theta+\gamma_{\max}^{1/2}\Theta+\|\gamma\|_{1,\infty}\|\bld u\|_{k+1})h^{k+2}.$$
This completes the proof of Theorem \ref{thm:ns-superror}.
\end{proof}

 \section{Hybridization}
 \label{sec:hybridization}
 In this section, we hybridize the $H(\mathrm{div})$-conforming HDG method \eqref{Hdiv-HDG-equations} by relaxing the
 $H(\mathrm{div})$-conformity of the velocity field via Lagrange multipliers;  similar treatment was used in \cite{CockburnSayasHDGStokes14}. 
 The resulting global linear system is a saddle point system for $(\uthat^h, \unhat^{h}, \bar{p}^{h})
 \in \Mth(0)\times \Mnh(0) \times \overline{Q}_{h}$, where
 \begin{subequations}
 \begin{align}
 \label{lambda-space}
  \Mnh(0):= & \{ \vvhat  \in \MMh(0) :\;\;  \mathrm{tr}_t (\vvhat)|_F = \boldsymbol{0},\;\;\forall F\in\Eh\},\\
 \label{av_pressure-space}
 \overline{Q}_{h} := & \{ q\in L^2(\Oh):\; q|_K \in \pol_{0}(K),\;\;\forall K\in\Oh\}.
  \end{align}
 \end{subequations} 
We show that $\uthat^h$ here is the same as that in (\ref{Hdiv-HDG-equations}), 
$\unhat^{h} = \trn (\bld u^{h})$ on $\Eh$, $\bar{p}^{h}$ is equal to average of $p^{h}$ on each element of $\Oh$.

Here we first relax $H(\mathrm{div})$-conformity of the velocity field in \eqref{Hdiv-HDG-equations} to obtain the
following result.
\begin{theorem}
  \label{thm:relax}
There exists a unique element $(\ml^h,\boldsymbol{u}^h,p_{\perp}^h, \bar{p}^{h},
{\uthat}^h, {\unhat}^h, \lambda^h)\in \GG_h\times \Vh\times Q^{\perp}_h \times \overline{Q}_{h}
\times  \Mth(0)\times \Mnh(0) \times \Mh^{\partial}$
such that the following weak formulation holds:
\begin{subequations}
\label{R-HDG-equations}
 \begin{alignat}{3}
 \label{R-HDG-equations-1}
 \bint{\ml^h}{\nu\,\mg^h}+\bint{\bld u^h}{\gradv \cdot (\nu\,\mg^h) }   
 - \bintEh{\uthat^h + \unhat^{h}}{ \nu\,\mg^h\, \n } & = 0, \\
 \label{R-HDG-equations-2}
 \bint{\nu\,\ml^h - (p_{\perp}^{h}+\bar{p}^{h})I_{d}}{\gradv \vv^h} +\bint{\gamma\, \bld u^h}{\vv^h} &
 \\
- \bintEh{\nu\,\ml^h \n - (p_{\perp}^{h}+\bar{p}^{h}) \n +\lambda^{h}\n}{\vv^{h}}
 &=(\bld f,\vv^h )_\Oh,\nonumber \\
 \label{R-HDG-equations-3}
 \bint{\divs\bld u^h}{q_{\perp}^h + \bar{q}^{h}} & = (g, q_{\perp}^h + \bar{q}^{h})_\Oh,\\
 \label{R-HDG-equations-4}
\bintEh{\nu\,\ml^h \n - (p_{\perp}^{h}+\bar{p}^{h}) \n +\lambda^{h}\n}{\vvhat_{t}^{h}+\vvhat_{n}^{h}} &= 0,\\
 \label{R-HDG-equations-5}
\bintEh{(\bld u^h - \unhat^{h})\cdot \n}{\mu^{h}}&=0,\\
 \label{R-HDG-equations-6}
(\bar{p}^h, 1)_\Oh &=0,
\end{alignat}
\end{subequations}
for all  $(\mg^h,\boldsymbol{v}^h,q_{\perp}^{h},\bar{q}^h,{\vthat}^h,\vvhat_{n}^{h},\mu^h)
\in \GG_h\times \Vh\times Q^{\perp}_h \times \overline{Q}_{h}
\times  \Mth(0)\times \Mnh(0) \times \Mh^{\partial}$, where 
\begin{align*}
Q_{h}^{\perp} := & \{ q\in L^2(\Oh):\; (q ,1)_{K}=0,\;\;\forall K\in\Oh\}, \\
\Mh^{\partial} := & \{ \mu \in L^{2}(\partial \Oh): \mu|_{\partial K}\in \pol_{k}(\partial K),\;\;\forall K\in\Oh\}, \\ 
\pol_{k}(\partial K) := & \{\mu \in L^{2}(\partial K): \mu |_{F} \in \pol_{k}(F),\;\;\forall F\in \mathcal{F}(K)\}.
\end{align*}

Moreover, if  $(\ml^h,\boldsymbol{u}^h,p_{\perp}^h, \bar{p}^{h},
{\uthat}^h, {\unhat}^h, \lambda^h)\in \GG_h\times \Vh\times Q^{\perp}_h \times \overline{Q}_{h}
\times  \Mth(0)\times \Mnh(0) \times \Mh^{\partial}$
is the numerical solution to the above equations, then
$(\ml^h,\boldsymbol{u}^h,p_{\perp}^h+\bar{p}^{h},{\uthat}^h)\in \GG_h\times \VVdiv(0)\times \Qh
\times  \Mth(0)$ is the only solution to \eqref{Hdiv-HDG-equations}.
 \end{theorem}

 Note that $\lambda^h\in \Mh^{\partial}$ is a quantity that approximates $0|_{\partial \Oh}$.
 \begin{proof}
Let
 $(\ml^h,\boldsymbol{u}^h,p_{\perp}^h, \bar{p}^{h}, {\uthat}^h, {\unhat}^h, \lambda^h)
 \in \GG_h\times \Vh\times Q^{\perp}_h \times \overline{Q}_{h}
\times  \Mth(0)\times \Mnh(0) \times \Mh^{\partial}$ be a numerical solution to equations \eqref{R-HDG-equations}.
We prove such numerical solution is unique and $(\ml^h,\boldsymbol{u}^h,p_{\perp}^h + \bar{p}^{h},\uthat^h)$ 
is the unique solution to equations \eqref{Hdiv-HDG-equations}.

  Since
  \begin{align*}
  (\bld u^h - \unhat^{h})\cdot \n |_{\partial K} \in \pol_{k}(\partial K)= \Mh^{\partial}(K), \quad 
  \forall K \in \Oh,
  \end{align*}
  we have $ \trn^{F} (\bld u^{h}) =  {\unhat}^h$ on any facet $F\in\Eh$ by equations \eqref{R-HDG-equations-5}.
  Hence, $\bld u^h\in \VVdiv(0)$. 

  By equation \eqref{R-HDG-equations-6}, we have $p_{\perp}^h + \bar{p}^{h} \in\Qh$.

Then, taking $\vv^h\in\VVdiv(0)$ in \eqref{R-HDG-equations-2}, 
 $\vvhat_{n}^{h} |_{F} = \trn^{F} (\vv^h)$ on any facet $F\in\Eh$ in \eqref{R-HDG-equations-4}, 
 and $q^h\in\Qh$ in \eqref{R-HDG-equations-3}, we have
   \[(\ml^h,\boldsymbol{u}^h,p_{\perp}^h + \bar{p}^{h},{\uthat}^h)\in \GG_h\times \VVdiv(0)\times \Qh
\times  \Mth(0)\] is the unique solution to equations \eqref{Hdiv-HDG-equations}.

Now, we only need to show the uniqueness of $\lambda^h$. If there are two $\lambda^{h}$, then by 
equation (\ref{R-HDG-equations-2}), their difference 
which we still call $\lambda^{h}$ satisfies 
\begin{align*}
\bintEh{\lambda^{h}}{\vv^{h}\cdot \n} = 0,\quad \forall \boldsymbol{v}^h \in \Vh.
\end{align*}
Since $\Mh^{\partial}(K) = \trn(\Vh(K))$ for any $K \in \Oh$, we have $\lambda^{h} = 0|_{\partial \Oh}$. 
So, $\lambda^{h}$ is also unique. This completes the proof.
 \end{proof}

Then, we identify local and global solvers.

Because of the lack of uniqueness of pressure in the Brinkman equations, we will keep 
$\bar{p}_{h}\in \overline{Q}_{h}$ as a separate unknown. 

Given $(\uthat, \unhat)\in \Mth(0)\times \Mnh(0)$, $\bld f\in L^2(\Oh)^d$, 
and $g\in L^2(\Oh)$, we consider the solution to the set of local problems in each element
$K\in\Oh$: find
\[
 (\ml^h,\bld u^h,p_{\perp}^h, \lambda^{h})\in \GG(K)\times \VV(K)\times \QQ^{\perp}(K) \times \Mh^{\partial}(K)
\]
such that
\begin{subequations}
\label{L-HDG-equations}
 \begin{alignat}{3}
 \label{L-HDG-equations-1}
 \bintKK{\ml^h}{\nu\,\mg^h}+\bintKK{ \bld u^h}{\gradv \cdot (\nu\,\mg^h)} 
& = \bintK{\uthat+ \unhat}{ \nu\,\mg^h\, \n}, \\
 \label{L-HDG-equations-2}
 -\bintKK{ \gradv \cdot (\nu\,\ml^h)-\nabla p_{\perp}^{h} - \gamma\, \bld u^h}{\vv^h} 
 - \bintK{\lambda^{h}\,\n}{\vv^h} & =(\bld f,\vv^h)_\Oh\\
 \label{L-HDG-equations-3}
 \bintKK{\divs\bld u^h}{q_{\perp}^h} & = (g, q_{\perp}^h)_\Oh,\\
 \label{L-HDG-equations-4}
 \bintK{(\bld u^{h} - \unhat)\cdot \n}{\mu^{h}} &= 0, 
\end{alignat}
\end{subequations}
for all  $(\mg^h,\boldsymbol{v}^h,q_{\perp}^h, \mu^{h})\in \GG(K)\times \VV(K)
\times \QQ_{\perp}(K) \times \Mh^{\partial}(K)$.

Unique solvability of this problem is a simple consequence of
unique solvability of the equations \eqref{R-HDG-equations}.

The solution to \eqref{L-HDG-equations} can be written as
\begin{align*}
 & (\ml^h,\bld u^h,p_{\perp}^h, \lambda^{h}) \\
 = & \left(\ml^h_{(\uthat,\unhat)},\bld u^h_{(\uthat,\unhat)},
 p^h_{\perp,(\uthat,\unhat)}, \lambda^{h}_{(\uthat,\unhat)}\right)
 +\left(\ml^h_{(\bld f,g)},\bld u^h_{(\bld f,g)},p^h_{\perp,(\bld f,g)},\lambda^{h}_{(\bld f, g)}\right)
\end{align*}
by considering separately the influence of $(\uthat,\uthat)$ and $(\bld f,g)$ in the solution. For example,
$\left(\ml^h_{(\uthat,\unhat)},\bld u^h_{(\uthat,\unhat)},
 p^h_{\perp,(\uthat,\unhat)}, \lambda^{h}_{(\uthat,\unhat)}\right)$ 
is the solution of \eqref{L-HDG-equations} when $(\bld f,g)=(\boldsymbol{0},0)$.

According to equations (\ref{R-HDG-equations-3},\ref{R-HDG-equations-4},\ref{R-HDG-equations-6}), 
the global (hybrid) problem is to find $(\uthat^{h}, \unhat^{h}, \bar{p}^{h}) \in \Mth(0)\times \Mnh(0) 
\times \overline{Q}_{h}$ such that 
\begin{subequations}
\label{G-HDG}
\begin{align}
& \bintEh{\nu\,\ml^h_{(\uthat^{h},\unhat^{h})} \n - (p_{\perp, (\uthat^{h},\unhat^{h})}^{h}+\bar{p}^{h}) \n
 +\lambda^{h}_{(\uthat^{h},\unhat^{h})}\n}{\vvhat_{t}^{h}+\vvhat_{n}^{h}} \\
 \nonumber 
&\qquad =  \bintEh{\nu\,\ml^h_{(\bld f, g)} \n - p_{\perp, (\bld f, g)}^{h} \n 
+\lambda^{h}_{(\bld f, g)}\n}{\vvhat_{t}^{h}+\vvhat_{n}^{h}}, \\
& \bint{\divs(\bld u^h_{(\uthat^{h},\unhat^{h})} + \bld u^{h}_{(\bld f, g)})}{\bar{q}^{h}} 
= (g, \bar{q}^{h})_\Oh, \\
& (\bar{p}^h, 1)_\Oh =0, 
\end{align}
\end{subequations}
for all $(\vvhat_{t}^{h}, \vvhat_{n}^{h}, \bar{q}^{h}) \in \Mth(0)\times \Mnh(0) 
\times \overline{Q}_{h}$. Again, unique solvability of this problem is a simple 
consequence of that for equations \eqref{R-HDG-equations}.
Moreover, we have the following characterization of the equations (\ref{G-HDG}). 
Its proof is trivial; see, e.g., \cite{CockburnSayasHDGStokes14}.
\begin{proposition}
The equations (\ref{G-HDG}) can be rewritten as 
\begin{align*}
A_{h}(\uthat^{h},\unhat^{h};\vvhat_{t}^{h},\vvhat_{n}^{h}) 
+ B_{h}(\vvhat_{n}^{h}; \bar{p}^{h})
= & F_{h} (\vvhat_{t}^{h},\vvhat_{n}^{h}),\\
B_{h}(\unhat^{h}; \bar{q}^{h}) = & 0,\\
(\bar{p}^h, 1)_\Oh = & 0,
\end{align*}
where 
\begin{subequations}
\begin{align}
A_{h}(\uthat^{h},\unhat^{h};\vvhat_{t}^{h},\vvhat_{n}^{h}) 
:= & \bint{\nu \ml^h_{(\uthat^{h},\unhat^{h})}}{\ml^h_{(\vvhat_{t},\vvhat_{n})}} 
+ \bint{\gamma \bld u^h_{(\uthat^{h},\unhat^{h})} }{\bld u^h_{(\vvhat_{t}^{h},\vvhat_{n}^{h})}},\\ 
B_{h}(\vvhat_{n}^{h}; \bar{p}^{h}) := & - \bintEh{\bar{p}^{h}}{\vvhat_{n}^{h}\cdot \n}, \\
F_{h} (\vvhat_{t}^{h},\vvhat_{n}^{h}) := & (\bld f, \bld u^h_{(\vvhat_{t}^{h},\vvhat_{n}^{h})})_{\Oh} 
- \bint{\nu \ml^h_{(\bld f,g)}}{\ml^h_{(\vvhat_{t},\vvhat_{n})}} \\
\nonumber
&\qquad - \bint{\gamma \bld u^h_{(\bld f,g)} }{\bld u^h_{(\vvhat_{t}^{h},\vvhat_{n}^{h})}}.
\end{align}
\end{subequations}
\end{proposition}


 \section{Numerical results}
 \label{sec:numerics}
In this section, we present two-dimensional numerical studies on both  rectangular and 
triangular meshes to validate the
theoretic results in Section \ref{sec:main}.

We use the Deal.II \cite{dealii84} software to implement the HDG method \eqref{Hdiv-HDG-equations} on rectangular meshes, and NGSolve \cite{Schoberl97,Schoberl16} on triangular meshes. Recall that our approximation spaces are given in Table \ref{table-example-m}.

The implementation on rectangular meshes use the hybridization discussed in Section \ref{sec:hybridization}; 
while the implementation on triangular meshes use NGSolve's built-in static condensation approach, see \cite{Schoberl16}. 

We present three numerical tests with a manufactured solution to 
validate our theoretic results in Section \ref{sec:main}.
For all the tests, the body forces $\bld f$ and $g$ are chosen such that
the exact solution $(\bld u, p)$ takes the following form:
\begin{align*}
 \bld u =&\; \left(\sin(2\,\pi x)\sin(2\,\pi y), \sin(2\,\pi x)\sin(2\,\pi y)\right)^T,\\
 p =&\; \sin(m\,\pi x)\sin(m\,\pi y), \text{ where $m$ is a fixed number.}
\end{align*}
We take $\nu = 1, \gamma=1$, and $m = 2$ for the first test, $\nu = 1$, $\gamma=1$, and $m = 20$ for the second test, and $\nu = 0.0001$, $\gamma=1$, and $m = 2$ for the third test.
The first two tests are in the Stokes-dominated regime, while the last test is in the Darcy-dominated regime.
The second test exam the effect of pressure regularity on the convergence of the velocity field.

In Table \ref{table-1}, 
we present the $L^2$-convergence rates for $\ml^h$, $\bld u^h$, $p^h$, and $\bld u^{*,h}$
for the HDG method \eqref{Hdiv-HDG-equations} with polynomial degree varying from $k=0$ to $k=3$
on rectangular meshes. The first level mesh  consists of 
$8\times 8$ congruent squares, and the consequent meshes are obtained by uniform refinements.

In Table \ref{table-1n0},
we present the same convergence study 
with polynomial degree varying from $k=1$ to $k=3$
on triangular meshes. The first level mesh  consists of 
$2\times 4\times 4$ congruent triangles, and the consequent meshes are obtained by uniform refinements.

In both  tables, $N_{ele}$ denotes the number of elements, $N_{global}$ denotes the
number of globally coupled degrees of freedom  and
$N_{local}$ denotes the number of local (static-condensed) degrees of freedom.

Here, the local postprocessing $\bld u^{*,h}\in \bpol_{k+1}(K)$ is defined element-wise 
by the
following set of equations:
 \begin{alignat*}{2}
  (\gradv \bld u^{*,h},\gradv \bld v)_K = &\; (\ml^h, \gradv \bld v)_K&&\quad \forall \bld v\in\bpol_{k+1}(K),\\
  (\bld u^{*,h},\bld w)_K = &\; (\bld u^h, \bld w)_K&&\quad \forall \bld w\in\bpol_{0}(K).
 \end{alignat*}
It is quite easy to show (c.f. \cite{Stenberg91,CockburnGopalakrishnanSayas10}) that $\bld u^{*,h}$ convergence with an order of $k+2-\delta_{0,k}$.

From the results for the first test in Table \ref{table-1}, we observe optimal convergence order of $k+1$ for all the three variables
$\ml^h, \bld u^h,$ and $p^h$, and superconvergence order of $k+2$ for the postprocessing $\bld u^{*,h}$.
The convergence results for $\ml^h, \bld u^h,$ and $p^h$ are in full agreements with the theoretic predictions in Corollary \ref{coro:ns-error} and
Theorem \ref{thm:ns-superror}.
The superconvergence for $\bld u^{*,h}$ is in agreement with the theoretic predictions
in Theorem \ref{thm:ns-superror} for $k\ge 1$, while the superconvergence of $\bld u^{*,h}$ for $k=0$ is not covered by our analysis in Theorem \ref{thm:ns-superror}. 

From the results for the second test in Table \ref{table-1}, we observe the same $L^2$-errors in
$\ml^h, \bld u^h,$ and $\bld u^{*,h}$  as the corresponding ones  in the first test.
This indicates  velocity error is independent of the pressure, in full agreement with the estimates in  Corollary \ref{coro:ns-error}.
We also observe the $L^2$-error for $p^h$ is significantly larger than that for the first test. It is clear that, in this test, convergence for pressure is not in the asymptotic regime yet.

From the results for the third test in Table \ref{table-1}, we observe similar convergence rates 
for all the variables as the first test. This indicates uniform stability of the proposed HDG method.

The convergence results on triangular meshes in Table \ref{table-1n0} are similar to that on rectangular meshes in Table \ref{table-1}.


\begin{table}[ht]
\caption{History of convergence for $H(\mathrm{div})$-conforming HDG method on square meshes.} 
\centering 
\resizebox{01.\columnwidth}{!}{%
\begin{tabular}{|c|c|c c|c c|c c|c c|c c|}
\hline
  &  mesh & \multicolumn{2}{c|}{$D.O.F.$}
          & \multicolumn{2}{c|}{$\|\ml - \ml^h \|_{\Oh}$}
          & \multicolumn{2}{c|}{$\|\bld u-\bld u^h\|_{\Oh}$}
          & \multicolumn{2}{c|}{$\|p - p^h \|_{\Oh}$}
          & \multicolumn{2}{c|}{$\|\bld u-\bld u^{*,h}\|_{\Oh}$}
\tabularnewline
$k$ & $N_{ele}$ & $N_{global}$ & $N_{local}$& error & order & error & order & error & order & error & order \tabularnewline
\hline
\multicolumn{12}{|c|}{First test. $\nu = 1,\gamma =1, m = 2$.} \\
\hline
\multirow{5}{2mm}{0} 
&     64&   288&    704&  2.393e+00& -& 1.622e-01& -& 4.133e-01& -& 5.398e-02& - \\
&    256&  1088&   2816&  1.224e+00& 0.97& 8.043e-02& 1.01& 1.300e-01& 1.67& 1.337e-02& 2.01 \\
&   1024&  4224&  11264&  6.157e-01& 0.99& 4.011e-02& 1.00& 4.782e-02& 1.44& 3.335e-03& 2.00 \\
&   4096& 16640&  45056&  3.083e-01& 1.00& 2.004e-02& 1.00& 2.108e-02& 1.18& 8.331e-04& 2.00
\tabularnewline
\hline
\multirow{5}{2mm}{1} 
&   64&     576&    1856& 4.951e-01& -& 1.829e-02& -& 1.178e-01& -& 6.955e-03& -\\
&  256&    2176&    7424& 1.286e-01& 1.94& 4.211e-03& 2.12& 1.559e-02& 2.92& 7.790e-04& 3.16\\
& 1024&    8448&   29696& 3.245e-02& 1.99& 1.026e-03& 2.04& 2.518e-03& 2.63& 9.367e-05& 3.06\\
& 4096&   33280&  118784& 8.131e-03& 2.00& 2.546e-04& 2.01& 5.171e-04& 2.28& 1.159e-05& 3.02
\tabularnewline
\hline
\multirow{5}{2mm}{2} 
&     64   &  864  &  3328 &5.810e-02& -& 1.399e-03& -& 1.281e-02& -& 7.069e-04& -\\
&    256   & 3264  & 13312 &7.352e-03& 2.98& 1.481e-04& 3.24& 9.173e-04& 3.80& 4.129e-05& 4.10\\
&   1024   &12672  & 53248 &9.223e-04& 2.99& 1.731e-05& 3.10& 7.743e-05& 3.57& 2.533e-06& 4.03\\
&   4096   &49920  &212992 &1.154e-04& 3.00& 2.122e-06& 3.03& 8.097e-06& 3.26& 1.575e-07& 4.01
\tabularnewline
\hline
\multirow{5}{2mm}{3} 
&    64&    1152&    5248& 5.598e-03& -& 9.147e-05& -& 1.740e-03& -& 6.264e-05& -\\
&   256&    4352&   20992& 3.600e-04& 3.96& 4.127e-06& 4.47& 9.163e-05& 4.25& 2.049e-06& 4.93\\
&  1024&   16896&   83968& 2.272e-05& 3.99& 2.222e-07& 4.21& 5.203e-06& 4.14& 6.492e-08& 4.98\\
&  4096&   66560&  335872& 1.424e-06& 4.00& 1.325e-08& 4.07& 3.112e-07& 4.06& 2.036e-09& 5.00
\tabularnewline
\hline
\multicolumn{12}{|c|}{Second test. $\nu = 1,\gamma =1, m = 20$.} \\
\hline
\multirow{5}{2mm}{0} 
&     64&   288&    704&  2.393e+00& -& 1.622e-01& -& 6.293e-01& -& 5.398e-02& - \\
&    256&  1088&   2816&  1.224e+00& 0.97& 8.043e-02& 1.01& 4.983e-01& 0.34& 1.337e-02& 2.01 \\
&   1024&  4224&  11264&  6.157e-01& 0.99& 4.011e-02& 1.00& 3.494e-01& 0.51& 3.335e-03& 2.00 \\
&   4096& 16640&  45056&  3.083e-01& 1.00& 2.004e-02& 1.00& 1.934e-01& 0.85& 8.331e-04& 2.00
\tabularnewline
\hline
\multirow{5}{2mm}{1} 
&   64&     576&    1856& 4.951e-01& -& 1.829e-02& -& 5.117e-01& -& 6.955e-03& -\\
&  256&    2176&    7424& 1.286e-01& 1.94& 4.211e-03& 2.12& 4.186e-01& 0.29& 7.790e-04& 3.16\\
& 1024&    8448&   29696& 3.245e-02& 1.99& 1.026e-03& 2.04& 1.631e-01& 1.36& 9.367e-05& 3.06\\
& 4096&   33280&  118784& 8.131e-03& 2.00& 2.546e-04& 2.01& 4.573e-02& 1.83& 1.159e-05& 3.02
\tabularnewline
\hline
\multirow{5}{2mm}{2} 
&     64   &  864  &  3328 &5.810e-02& -& 1.399e-03& -& 4.917e-01& -& 7.069e-04& -\\
&    256   & 3264  & 13312 &7.352e-03& 2.98& 1.481e-04& 3.24& 2.722e-01& 0.85& 4.129e-05& 4.10\\
&   1024   &12672  & 53248 &9.223e-04& 2.99& 1.731e-05& 3.10& 5.209e-02& 2.39& 2.533e-06& 4.03\\
&   4096   &49920  &212992 &1.154e-04& 3.00& 2.122e-06& 3.03& 7.240e-03& 2.85& 1.575e-07& 4.01
\tabularnewline
\hline
\multirow{5}{2mm}{3} 
&    64&    1152&    5248& 5.598e-03& -& 9.147e-05& -& 4.744e-01& -& 6.264e-05& -\\
&   256&    4352&   20992& 3.600e-04& 3.96& 4.127e-06& 4.47& 1.362e-01& 1.80& 2.049e-06& 4.93\\
&  1024&   16896&   83968& 2.272e-05& 3.99& 2.222e-07& 4.21& 1.252e-02& 3.44& 6.492e-08& 4.98\\
&  4096&   66560&  335872& 1.424e-06& 4.00& 1.325e-08& 4.07& 8.610e-04& 3.86& 2.036e-09& 5.00
\tabularnewline
\hline
\multicolumn{12}{|c|}{Third test. $\nu = 0.0001,\gamma =1, m = 2$.} \\
\hline
\multirow{5}{2mm}{0} 
&  64 &   288 &   704 &2.399e+00&  -  & 1.621e-01& -& 1.567e-01& -&5.329e-02& - \\
& 256 &  1088 &  2816 &1.226e+00& 0.97& 8.039e-02& 1.01& 7.970e-02& 0.98&1.313e-02&2.02 \\
&1024 &  4224 & 11264 &6.160e-01& 0.99& 4.011e-02& 1.00& 4.002e-02& 0.99& 3.268e-03 &2.01 \\
&4096 & 16640 & 45056 &3.083e-01& 1.00& 2.004e-02& 1.00& 2.003e-02& 1.00& 8.164e-04 &2.00 
\tabularnewline
\hline
\multirow{5}{2mm}{1} 
 &  64  &   576  &  1856 &3.779e-01 &- &1.679e-02 &- &2.967e-02 &- &6.192e-03 &- \\
 &  256  &  2176  &  7424 &9.967e-02 &1.92 &4.096e-03 &2.04 &7.556e-03 &1.97 &7.509e-04 &3.04 \\
 & 1024  &  8448  & 29696 &2.761e-02 &1.85 &1.020e-03 &2.01 &1.898e-03 &1.99 &9.297e-05 &3.01 \\
 & 4096  & 33280  &118784 &7.630e-03 &1.86 &2.544e-04 &2.00 &4.750e-04 &2.00 &1.157e-05 &3.01 
\tabularnewline
\hline
\multirow{5}{2mm}{2} 
  &   64 &    864  &  3328 &4.844e-02 &- &1.223e-03 &- &3.755e-03 &- &6.990e-04 &- \\
   &    256 &   3264  & 13312 &6.177e-03 &2.97 &1.399e-04 &3.13 &4.773e-04 &2.98 &4.215e-05 &4.05 \\
    &   1024 &  12672  & 53248 &8.198e-04 &2.91 &1.708e-05 &3.03 &5.992e-05 &2.99 &2.571e-06 &4.04\\ 
    &   4096 &  49920  &212992 &1.099e-04 &2.90 &2.118e-06 &3.01 &7.498e-06 &3.00 &1.584e-07 &4.02 
\tabularnewline
\hline
\multirow{5}{2mm}{3} 
&     64&    1152&   5248 &4.973e-03 &- &7.545e-05 &- &3.567e-04 &- &6.160e-05 &- \\
&    256&    4352&   20992& 3.248e-04&3.94 &3.766e-06 &4.32 &2.264e-05 &3.98 &2.038e-06 &4.92 \\
&   1024&   16896&   83968& 2.136e-05& 3.93& 2.173e-07& 4.12 &1.420e-06& 3.99 &6.486e-08& 4.97 \\
&   4096&   66560&  335872& 1.390e-06& 3.94& 1.322e-08& 4.04 &8.885e-08& 4.00 &2.035e-09& 4.99 
\tabularnewline
\hline
\end{tabular}}
\label{table-1} 
\end{table}


\begin{table}[ht]
\caption{History of convergence for $H(\mathrm{div})$-conforming HDG method on triangular meshes.} 
\centering 
\resizebox{01.\columnwidth}{!}{%
\begin{tabular}{|c|c|c c|c c|c c|c c|c c|}
\hline
  &  mesh & \multicolumn{2}{c|}{$D.O.F.$}
          & \multicolumn{2}{c|}{$\|\ml - \ml^h \|_{\Oh}$}
          & \multicolumn{2}{c|}{$\|\bld u-\bld u^h\|_{\Oh}$}
          & \multicolumn{2}{c|}{$\|p - p^h \|_{\Oh}$}
          & \multicolumn{2}{c|}{$\|\bld u-\bld u^{*,h}\|_{\Oh}$}
\tabularnewline
$k$ & $N_{ele}$ & $N_{global}$ & $N_{local}$& error & order & error & order & error & order & error & order \tabularnewline
\hline
\multicolumn{12}{|c|}{First test. $\nu = 1, \gamma = 1, m = 2$.} \\
\hline
\multirow{5}{2mm}{1} 
&32      &256     &   555 &1.567e+00 &-  &8.253e-02 &-  &5.144e-01 &-  &5.985e-02 &-\\
&128     &960     &  2203 &3.378e-01 &2.21  &3.220e-02 &1.36  &1.158e-01 &2.15  &6.449e-03 &3.21\\
&512     &3712    &  8763 &8.757e-02 &1.95  &8.073e-03 &2.00  &2.712e-02 &2.09  &8.455e-04 &2.93\\
&2048    &14592   & 34939 &2.213e-02 &1.98  &2.018e-03 &2.00  &6.559e-03 &2.05  &1.073e-04 &2.98\\
&8192    &57856   &139515 &5.550e-03 &2.00  &5.045e-04 &2.00  &1.615e-03 &2.02  &1.348e-05 &2.99
\tabularnewline
\hline
\multirow{5}{2mm}{2} 
&      32   &   368  &   1163 &9.679e-02 &-  &3.553e-02 &-  &4.949e-02 &-  &2.407e-03 &-\\
&     128   &  1376  &   4635 &3.471e-02 &1.48  &3.432e-03 &3.37  &1.183e-02 &2.07  &4.712e-04 &2.35\\
&     512   &  5312  &  18491 &4.381e-03 &2.99  &4.359e-04 &2.98  &1.488e-03 &2.99  &2.964e-05 &3.99\\
&    2048   & 20864  &  73851 &5.488e-04 &3.00  &5.472e-05 &2.99  &1.862e-04 &3.00  &1.854e-06 &4.00\\
&    8192   & 82688  & 295163 &6.864e-05 &3.00  &6.847e-06 &3.00  &2.325e-05 &3.00  &1.159e-07 &4.00
\tabularnewline
\hline
\multirow{5}{2mm}{3} 
&      32 &     480  &   1995 &3.551e-02 &-  &1.557e-03 &-  &2.159e-02 &-  &1.760e-03 &-\\
&     128 &    1792  &   7963 &1.815e-03 &4.29  &.245e-04 &2.79  &9.946e-04 &4.44  &4.237e-05 &5.38\\
&     512 &    6912  &  31803 &1.172e-04 &3.95  &1.418e-05 &3.99  &6.099e-05 &4.03  &1.356e-06 &4.97\\
&    2048 &   27136  & 127099 &7.398e-06 &3.99  &8.883e-07 &4.00  &3.774e-06 &4.01  &4.266e-08 &4.99\\
&    8192 &  107520  & 508155 &4.638e-07 &4.00  &5.555e-08 &4.00  &2.348e-07 &4.01  &1.336e-09 &5.00
\tabularnewline
\hline
\multicolumn{12}{|c|}{Second test. $\nu = 1, \gamma = 1, m = 20$.} \\
\hline
\multirow{5}{2mm}{1} 
&      32&      256&      555 &1.582e+00 &-  &8.376e-02 &-  &1.022e+00 &-  &6.085e-02 &-\\
&     128&      960&     2203 &3.652e-01 &2.12  &3.256e-02 &1.36  &6.395e-01 &0.68  &7.492e-03 &3.02\\
&     512&     3712&     8763 &8.758e-02 &2.06  &8.073e-03 &2.01  &3.010e-01 &1.09  &8.457e-04 &3.15\\
&    2048&    14592&    34939 &2.213e-02 &1.98  &2.018e-03 &2.00  &1.091e-01 &1.46  &1.073e-04 &2.98\\
&    8192&    57856&   139515 &5.550e-03 &2.00  &5.045e-04 &2.00  &3.012e-02 &1.86  &1.348e-05 &2.99
\tabularnewline
\hline
\multirow{5}{2mm}{2} 
&      32&      368&     1163& 1.813e-01& -&  3.599e-02& -&  6.015e-01& -&  5.208e-03& -\\
&     128&     1376&     4635& 3.471e-02& 2.39&  3.432e-03& 3.39&  4.004e-01& 0.59&  4.715e-04& 3.47\\
&     512&     5312&    18491& 4.381e-03& 2.99&  4.359e-04& 2.98&  1.741e-01& 1.20&  2.964e-05& 3.99\\
&    2048&    20864&    73851& 5.488e-04& 3.00&  5.472e-05& 2.99&  3.076e-02& 2.50&  1.854e-06& 4.00\\
&    8192&    82688&   295163& 6.864e-05& 3.00&  6.847e-06& 3.00&  4.192e-03& 2.88&  1.159e-07& 4.00
\tabularnewline
\hline
\multirow{5}{2mm}{3} 
&      32&      480&     1995 &4.193e-02 &-  &1.793e-03 &-  &5.592e-01 &-  &1.907e-03 &-\\
&     128&     1792&     7963 &1.815e-03 &4.53  &2.245e-04 &3.00  &3.068e-01 &0.87  &4.237e-05 &5.49\\
&     512&     6912&    31803 &1.172e-04 &3.95  &1.418e-05 &3.99  &3.201e-02 &3.26  &1.356e-06 &4.97\\
&    2048&    27136&   127099 &7.398e-06 &3.99  &8.883e-07 &4.00  &1.507e-03 &4.41  &4.266e-08 &4.99\\
&    8192&   107520&   508155 &4.638e-07 &4.00  &5.555e-08 &4.00  &6.589e-05 &4.52  &1.336e-09 &5.00
\tabularnewline
\hline
\multicolumn{12}{|c|}{Third test. $\nu = 0.0001, \gamma = 1, m = 2$.} \\
\hline
\multirow{5}{2mm}{1} 
 &     32  &    256  &    555& 1.436e+00& -&  7.825e-02& -&  3.891e-02& -&  5.242e-02& -\\
 &    128  &    960  &   2203& 3.932e-01& 1.87&  3.013e-02& 1.38&  1.949e-02& 1.00&  8.254e-03& 2.67\\
 &    512  &   3712  &   8763& 1.241e-01& 1.66&  7.719e-03& 1.96&  4.951e-03& 1.98&  1.497e-03& 2.46\\
 &   2048  &  14592  &  34939& 3.414e-02& 1.86&  1.962e-03& 1.98&  1.243e-03& 1.99&  2.153e-04& 2.80\\
 &   8192  &  57856  & 139515& 7.387e-03& 2.21&  4.987e-04& 1.98&  3.110e-04& 2.00&  2.215e-05& 3.28
\tabularnewline
\hline
\multirow{5}{2mm}{2} 
 &     32  &    368  &   1163& 3.442e-01& -& 3.288e-02& -&  2.266e-02& -&  1.065e-02& -\\
 &    128  &   1376  &   4635& 6.978e-02& 2.30&  3.132e-03& 3.39&  2.169e-03& 3.39&  9.514e-04& 3.48\\
 &    512  &   5312  &  18491& 1.042e-02& 2.74&  4.049e-04& 2.95&  2.748e-04& 2.98&  7.304e-05& 3.70\\
 &   2048  &  20864  &  73851& 1.085e-03& 3.26&  5.285e-05& 2.94&  3.447e-05& 3.00&  4.209e-06& 4.12\\
 &   8192  &  82688  & 295163& 9.842e-05& 3.46&  6.770e-06& 2.96&  4.313e-06& 3.00&  1.938e-07& 4.44
\tabularnewline
\hline
\multirow{5}{2mm}{3} 
&      32&      480 &    1995& 3.231e-02& -&  1.545e-03& -&  6.370e-04& - & 1.717e-03& -\\
&     128&     1792 &    7963& 6.311e-03& 2.36&  1.928e-04& 3.00&  2.490e-05& 4.68 & 6.757e-05& 4.67\\
&     512&     6912 &   31803& 3.905e-04& 4.01&  1.295e-05& 3.90&  1.348e-06& 4.21 & 2.318e-06& 4.87\\
&    2048&    27136 &  127099& 1.768e-05& 4.46&  8.583e-07& 3.91&  8.055e-08& 4.07 & 6.261e-08& 5.21\\
&    8192&   107520 &  508155& 7.284e-07& 4.60&  5.500e-08& 3.96&  4.974e-09& 4.02 & 2.111e-09& 4.89
\tabularnewline
\hline
\end{tabular}}
\label{table-1n0} 
\end{table}

 \section{Conclusion}
 \label{sec:conclusion}
 We present and analyze a class of parameter-free superconvergent $H(\mathrm{div})$-conforming HDG method on both  simplicial and rectangular meshes for the Brinkman equations.
 Numerical results in two dimensions are presented to validate the theoretic findings.

 \section*{Acknowledgements}
G. Fu would like to thank Matthias Maier from the University of Minnesota for
providing the general framework of the HDG code in deal.II and for many helpful discussions on numerical computations with deal.II.
He would also like to thank Christoph Lehrenfeld from University of G{\"o}ttingen  for
many helpful discussions and hands-on tutorials on numerical computation using NGSolve's python interface.

 \section*{Appendix: Proof of Lemma \ref{lemma:key}}
In this Appendix, we prove Lemma \ref{lemma:key}.
We use the following result, whose proof comes directly from Lemma \ref{lemma:projection-b} and the usual scaling argument.
\begin{lemma}
\label{lemma:1}
  Given $(\mr^h, \zzhat^h)\in \GG(K)\times \MM(\dK)$ where
 \[
  \MM(\dK):=\{\vvhat\in L^2(\dK)^d:\;\vvhat|_F\in \MM(F)\;\;\forall F\in\mathcal{F}(K)\},
 \]
there exists a unique function $\bld w^h\in \VV(K)$ such that
\begin{alignat*}{2}
 ({\bld w^h},{\vv^h})_K = &\; ({\divv \mr^h},{\vv^h})_K&&\;\;\forall \vv^h\in \divv\GG(K),\\
 \bintK{\trn(\bld w^h)}{\trn(\vvhat)} = &\; \bintK{\trn(\zzhat^h)}{\trn(\vvhat)}&&\;\;\forall \vvhat^h\in \MM(\dK).
\end{alignat*}
Moreover, there exists a constant $C$ only depending on the shape-regularity of the element $K$ such that
\begin{align}
 \label{estimate-2}
 \|\bld w^h\|_K\le C\left(\|\divv \mr^h\|_K^2 + \sum_{F\in\mathcal{F}(K)}h_F\|\trn(\zzhat^h)\|_F^2\right)^{1/2}
\end{align}
\end{lemma}

\subsection*{Proof of Lemma \ref{lemma:key}}
\begin{proof}
We only prove the existence and uniqueness of the function $\mr^h\in \GG(K)$ satisfying equations \eqref{g-proj} on the reference element $K=\Kbar$,
the result on an affine-mapped element $K$ can be easily obtained from that on the reference element
(cf. \cite[Chapter 2]{BoffiBrezziFortin13}), and the estimate
\eqref{estimate-1} is a direct consequence of the usual scaling argument and equivalence of norms on finite-dimensional spaces.

We first show that \eqref{g-proj} define a square system.
We use the concept of an M-decomposition \cite{CockburnFuSayas16,CockburnFu17a,CockburnFu17b} to prove it.

By the choice of $\GG^\mathrm{row}(K)$ in Table \ref{table-example-m}, we have
the pair $ \GG^\mathrm{row}(K)\times \pol_k(K)$
admits an M-decomposition with the trace space
\[
 M(\dK):=\{\wwhat\in L^2(\dK):\;\;\wwhat|_F\in\pol_k(F)\;\;\forall F\in\mathcal{F}(K)\}.
\]
Hence,
\begin{align*}
  \dim\GG^\mathrm{row}(K)+
  \dim\pol_k(K) = &\;
  \dim\GG^\mathrm{row}_\mathrm{sbb}(K)
  +
  \dim\divs\GG^\mathrm{row}(K)\\
&\;  +\dim\grads \pol_k(K)
  +\dim M(\dK).
\end{align*}
Here
$
 \GG^\mathrm{row}_\mathrm{sbb}(K):=
 \{
 \vv\in\GG^\mathrm{row}(K):\;\divs \vv=0,\; \trn(\vv)=0 \text{ on }\dK
 \}.
$
This immediately implies that
\begin{align}
\label{dim-1}
  \dim\GG(K)+
  \dim\pol_k(K)^d = &\;
  \dim\GG_\mathrm{sbb}(K)
  +
  \dim\divv\GG(K)\\
&\;  +\dim\gradv \pol_k(K)^d
  +\dim \MM(\dK).\nonumber
\end{align}

By Lemma \ref{lemma:projection-b}, we have
\[
\dim \VV(K) =  \dim\divv\GG(K)+\dim \trn(\MM(\dK)).
\]
Combing the above equality with \eqref{dim-1} and reordering the terms, we get
\begin{align}
\label{dim-2}
  \dim\GG(K) = &\;
  \dim\GG_\mathrm{sbb}(K)
   +\dim \trt(\MM(\dK))
  \\
&\; +\dim\VV(K)-  \dim\pol_k(K)^d +\dim\gradv \pol_k(K)^d.\nonumber
\end{align}

Since it is trivial to prove that
\[
 \dim\VV(K)-  \dim\pol_k(K)^d +\dim\gradv \pol_k(K)^d = \dim \gradv \VV(K)
\]
for the vector space $\VV(K)$ in Table \ref{table-example-m}, we conclude that equations
\eqref{g-proj} is indeed a square system.
Hence, we are left to prove the uniqueness.

To this end, we take $\bld z^h=0, \zzhat^h=0$ in \eqref{g-proj}.
By \eqref{g-proj-2}, we have
\begin{align}
\label{t-t}
 \trt(\mr^h\n)=0.
\end{align}

By \eqref{g-proj-1}, we have, for all $\vv\in\VV(K)$,
\begin{align*}
0=  (\mr^h, \gradv \vv)_K
=&\; -(\divv\mr^h, \vv)_K
+\bintK{\trn(\mr^h\n)}{\trn(\vv)}
+\bintK{\trt(\mr^h\n)}{\trt(\vv)}\\
=&\;
-(\divv\mr^h, \vv)_K
+\bintK{\trn(\mr^h\n)}{\trn(\vv)}.
\end{align*}
Then, by Lemma \ref{lemma:1}, there exists a function $\vv\in\VV(K)$ such that
\[-(\divv\mr^h, \vv)_K
+\bintK{\trn(\mr^h\n)}{\trn(\vv)}
=(\divv\mr^h, \divv\mr^h)_K
+\bintK{\trn(\mr^h\n)}{\trn(\mr^h\n)}.
\]

Hence, $\divv\mr^h=0$ and $\trn(\mr^h\n)=0$.
This implies that $\mr^h\in \GG_{\mathrm{sbb}}(K)$. Then, taking
$\mg^h :=\mr^h\in \GG_{\mathrm{sbb}}(K)$ in \eqref{g-proj-1},  we conclude that $\mr^h=0$.

This conclude the proof of Lemma \ref{lemma:key}.
\end{proof}


\bibliographystyle{siam}

\end{document}